\documentclass[11pt,reqno]{amsart}
\usepackage{amsmath}
\usepackage{amssymb}
\usepackage{latexsym}
\usepackage{mathrsfs}
\def\squarebox#1{\hbox to #1{\hfill\vbox to #1{\vfill}}}

\DeclareMathOperator{\grad}{grad}
\newcommand{\1}{{\bold 1}}

\theoremstyle{plain}

\newtheorem{thm}{Theorem}
\newtheorem{cor}{Corollary}
\newtheorem{lem}{Lemma}
\newtheorem{rem}{Remark}
\newtheorem{prop}{Proposition}
\newtheorem{defn}{Definition}

\renewcommand{\thelem}{\thesection.\arabic{lem}}

\renewcommand{\theequation}{\thesection.\arabic{equation}}

\def\pu{\partial_t u}
\def\la{\langle}
\def\ra{\rangle}
\def\12{\frac{1}{2}}

\def\phi{\varphi}
\def\epsilon{\varepsilon}
\def\kappa{\varkappa}
\def\ep{\epsilon}

\def\lr#1{\langle{#1}\rangle}
\def\dif{\partial}

\def\al{\alpha}
\def\be{\beta}
\def\gam{\gamma}
\def\R{{\mathbb R}}

\numberwithin{equation}{section}

\begin{document}

\def\R {{\mathbb{R}}}
\def\N {{\mathbb{N}}}
\def\C {{\mathbb{C}}}
\def\Z {{\mathbb{Z}}}
\def\T{{\mathbb T}}
\def\Q{{\mathbb Q}}
\def\SP{{\mathbb S}}
\def\mc{{\mathcal H}}
\def\1b{{\mathbb I}}
\def\pu{\partial_t u}
\def\pt{\partial_t}
\def\pa{\partial}

 \vbadness=10000
 \hbadness=10000

\def\phi{\varphi}
\def\epsilon{\varepsilon}
\def\kappa{\varkappa}
\def\eT{e^{-\lambda T}}
\def\ii{{\bf i}}
\def\hh{\hat{x}}
\def\hhx{\hat{\xi}}
\def\12{\frac{1}{2}}
\def\Rc{{\mathcal R}}
\def\tE{\tilde{E}}
\def\ep{\epsilon}
\def\la{\langle}
\def\ra{\rangle}
\def\co{{\mathcal O}}
\def\pa{\partial}

\def\Qed{\qed\par\medskip\noindent}

\title[Cauchy problem for hyperbolic operators ]{Cauchy problem for
  effectively hyperbolic operators with triple characteristics} 
\date{}
\author[T. Nishitani] {Tatsuo Nishitani}

\address{Departement of Mathematics, Osaka University, Machikaneyama 1-1, Toyonaka 560-0043, Japan}
\email{nishitani@math.sci.osaka-u.ac.jp}
\author[V. Petkov]{Vesselin Petkov}
\address{Institut de Math\'ematiques de Bordeaux, 351,
Cours de la Lib\'eration, 33405  Talence, France}
\email{petkov@math.u-bordeaux.fr}

%    \subjclass is required.%\subjclass {Primary 35L30, Secondary 35L25}
%\maketitle
\begin{abstract} We study the Cauchy problem for effectively hyperbolic operators $P$ with principal symbol $p(t, x,\tau,\xi)$ having triple characteristics on $t = 0$. Under a condition (E) we show that such operators are strongly hyperbolic, that is the Cauchy problem is well posed for $p(t, x,D_t, D_x) + Q(t, x, D_t, D_x)$ with arbitrary 
lower order term $Q$. The proof is based on energy estimates with weight $t^{-N}$ for a first order pseudo-differential system, where $N$ depends on lower order terms. For our analysis we construct a non-negative definite symmetrizer $S(t)$ and we prove a version of Fefferman-Phong type inequality for ${\mathsf{Re}}\, (S(t)U, U)_{L^2(\R^n)}$ with a lower bound $-C t^{-1}\|\lr{D}^{-1}U\|_{L^2(\R^n)}$.

\end{abstract}
%\subjclass[2010]{Primary 35L30, Secondary 35L25}
\vspace{0.5cm}
\maketitle
\section{Introduction}
Consider a differential operator
 $$P(t, x, D_t, D_x) = \sum _{\alpha + |\beta| \leq m} c_{\alpha, \beta} (t, x) D_t^{\alpha} D_x^{\beta},\;\; D_t = -i\partial_t,\;\; D_{x_j} = -i\partial_{x_j}$$ 
 of order $m$ with smooth coefficients $c_{\alpha, \beta}(t,x),\: t \in \R,\:x \in \R^n.$ Denote by 
$$p(t, x, \tau, \xi) = \sum_{\alpha + |\beta| = m} c_{\alpha, \beta} (t, x) \tau^{\alpha} \xi^{\beta}$$
 the principal symbol of $P$.
Let $U \subset \R^{n + 1}$ be an open set and let 
$$U_{\eta}^{-} = U \cap \{t \leq \eta\},\quad U_{\eta}^{+} = U \cap \{t \geq \eta\},\quad U_{\eta}=U\cap\{t=\eta\}.
$$
 We say that $P$ is hyperbolic with respect to $N_0 =(1,0,\ldots,0)$ at $(t_0, x_0)$ if
 \begin{itemize}
 \item[\rm(i)]  $ p(t_0, x_0, N_0) \neq 0$, 
\item[\rm(ii)]
the equation 
\begin{equation} \label{eq:1.1}
p(t_0, x_0, \tau, \xi) = 0
\end{equation}
 with respect to $\tau $ has only real roots  $\tau = \lambda_j(t_0, x_0, \xi)$ for all $\xi \in \R^n.$
 \end{itemize}
\begin{defn} We say that the Cauchy problem
\begin{equation} \label{eq:1.2}
Pu = f\quad  {\rm in}\quad U_T^{+} ,\quad D_t^ju\big|_{U_T}=0\quad (j=0,\ldots,m-1)
\end{equation}
is well posed in $U_T^{+}$ if
\begin{itemize}
\item[\rm(i)] (existence) for every $f \in C_0^{\infty}(U_T^{+})$ there
exists a solution $u \in C^m(U_T^{+})$ satisfying \eqref{eq:1.2}.\\
\item[\rm(ii)] (uniqueness) if $u \in C^m(U_T^{+})$ satisfies \eqref{eq:1.2}, then for every $s, s> T,$ if $Pu = 0$ in $U_s^{-}$, then $u = 0$ in $U_s^{-}$.
\end{itemize}
\end{defn}
Let $\Omega\subset \R^{n+1}$ be an open set and let $G =  \Omega \cap \{0 \leq t \leq T\}$.
A necessary condition for the well-posedness of the Cauchy problem  in $G$ is the hyperbolicity of the operator $P$ at every point  $(t, x) \in G.$
\begin{defn} We say that the operator $P$ with principal symbol $p_m$ is strongly hyperbolic in $G$ if for every point $z_0 = (t_0, x_0) \in G$ there exist a neighborhood $U$ of $z_0$  and $T_0 \geq 0$ ($T_0 < t_0$ if $t_0 > 0$ and $T_0= 0$ if $t_0 = 0$) such that \eqref{eq:1.2} is well posed in $(U\cap\{t<T\})^{+}_s$ for every $T_0 \leq s < T$ and every operator $L = p(t, x, D_t, D_x) + Q_{m-1}(t, x, D_t, D_x)$, $Q_{m-1}$ being any  operator of order less or equal to $m-1$.
\end{defn}
 Set $x = (x_0, x_1,\ldots, x_n)$, $\xi = (\xi_0, \xi_1,\ldots,\xi_n)=(\xi_0,\xi')$ with $x_0 = t,\: \xi_0 = \tau$  and consider the fundamental matrix $F_{p}$ of the principal symbol $p$

$$F_{p} (x, \xi) = \left(\begin{matrix} p_{\xi, x}(x, \xi) & &  p_{\xi, \xi}(x, \xi)\\
-p_{x, x}(x, \xi) & &  - p_{x,
  \xi}(x,\xi)\end{matrix}\right),
$$
introduced in \cite{IP}.
 If $P$ is hyperbolic in $G$ and $z$ is a critical point of $p$, then $F_{p}(z)$ has at most two non-vanishing real simple eigenvalues $\mu$ and $-\mu$ and all other eigenvalues are contained in $i \R$. 
It was also proved in \cite{IP} a necessary condition for $P$ to be strongly hyperbolic in $G$. Namely this condition says that at every  critical point $z$ of $p$  the fundamental matrix $F_{p}(z)$ has two non-zero real eigenvalues. Moreover, if $P$ is strongly hyperbolic in $G$ then for $(x,\xi') \in (\Omega\cap\{0<x_0<T\})  \times (\R^n \setminus \{0\})$ the multiplicities of the roots of \eqref{eq:1.1} are not greater than two, and for $(x,\xi') \in (\Omega\,\cap\{x_0=0\}) \times (\R^{n}\setminus \{0\})$ or for $(x,\xi') \in (\Omega\cap\{x_0=T\}) \times (\R^{n} \setminus \{0\})$ these multiplicities are not greater than three (see \cite{IP} for more details).
\begin{defn}
A hyperbolic operator is called {\it effectively hyperbolic}, if the matrix $F_{p}(z)$ has non-vanishing real eigenvalues at every critical point $z$ of $p$.
\end{defn}
 It was conjectured \cite{IP} that a hyperbolic operator with at most triple characteristics is strongly hyperbolic if and only if it is effectively hyperbolic. For operators with double characteristics the sufficiency of this condition has been proved by Iwasaki \cite{Iw1}, \cite{Iw2} and by the first author \cite{N1} (see also \cite{N2} for another proof). Some results for special class of operators have been obtained by Oleinik \cite{O}, Ivrii \cite{I}, H\"ormander \cite{H1} and Melrose \cite{M}.\\

The analysis of effectively hyperbolic operators with triple characteristics is more complicated. As we mentioned above such operators could have triple characteristics only for $t = 0$ or $t = T$. If this happens for $t = 0$, the operator is hyperbolic only for $t \geq 0$ and for $t < 0$ it has complex characteristics. In \cite{I} it was investigated the case when the principal symbol $p_3$ is smoothly factorized
$$p_3= ((\tau - c(t, x, \xi))^2- q(t, x,\xi)) (\tau - \lambda(t, x, \xi))$$
with smooth root $\lambda(t, x,\xi)$ and smooth $c(t, x,\xi)$ and $q(t, x, \xi) \geq 0.$
The proof is based  on  construction of a parametrix for the factors and for this the factorization is crucial. On the other hand, there are simple examples of effectively hyperbolic operators with triple characteristics whose principal symbol cannot be  factorized (see \cite{BBP}).\\

Let $U \subset \R^n$ be an open domain and let $G = [0, T] \times U$. In this paper we study an effectively hyperbolic operator $P$ of third order 
\begin{equation} 
\label{1.3}
\begin{split}
P  =  D_t^3 +   q_1(t, x, D_x)D_t^2 +q_2(t, x, D_x)D_t + q_3(t, x, D_x) \\
+  r_2(t,x, D_x) + r_1(t, x, D_x) D_t + r_0(t, x)D_t^2 \\ 
+ m_1(t, x, D_x) + m_0(t, x)D_t + c_0(t, x)
\end{split}
\end{equation}
with triple characteristics points lying on $t = 0$. Here $q_j,\: j = 1,2,3$ are operators of order $j$ and $r_j, m_j$ are operators of order $j$. Let $p(t, x, \tau, \xi)$ be the principal symbol of $P$. Consider the set
$$\Sigma_3 = \{\zeta=  (t, x, \tau, \xi):\: \frac{\pa^j p}{\pa \tau^j}(\zeta) = 0, \: j = 0,1,2\} \subset \{(0, x,\tau, \xi)\in T^*(G)\setminus \{0\}\}$$ 
 of triple characteristics points of $P$ lying on $t = 0$.  
Then $P$ is effectively hyperbolic if and only if (see \cite{IP}) we have
$$\frac{\pa^2 p}{\pa t \pa \tau}\big\vert_{\Sigma_3}  < 0.$$ 
Introduce the symbols 
$$\Delta_1  = 27 q_3 - 9q_1 q_2 + 2 q_1^3,\;\; \Delta_0 = q_1^2 - 3 q_2, \;\; \Delta =  -\frac{1}{27 \la \xi \ra^6} (\Delta_1^2 - 4 \Delta_0^3),$$
where $\la \xi \ra = ( 1 + |\xi|^2)^{1/2}.$
The symbol $\Delta \la \xi \ra ^6$ is the {\it discriminant} of the equation $p = 0$ with respect to $\tau$ and we have
three real roots for $t \geq 0$ if and only if $\Delta \geq 0.$ The
symbol $\Delta_0$ is the {\it discriminant} of the equation
$\pa_{\tau} p = 3 \tau^2 + 2 q_1 \tau + q_2 = 0$ with respect to $\tau$ and if we have a
triple root at $\rho = (0, x_0, \xi)$, we get $\Delta(\rho) =\Delta_0(\rho) = 0.$

After a change of variables preserving $t$, an effectively hyperbolic operators of third order with triple characteristics points lying on $t = 0$ can be transformed into  a pseudo-differential operator $P$ with principal  symbol (see \cite{BBP})
\begin{equation} \label{eq:1.4}
p(t, x, \tau, \xi) = \tau^3 -(t + \alpha(x,\xi))\tilde{a}(t, x, \xi)\la \xi \ra^2\tau  + b(t, x, \xi)\la \xi \ra ^3,\;\; t \in [0, T], \;\; (x, \xi) \in T^* (U),
\end{equation}
where   $\alpha(x, \xi),\: \tilde{a}(t, x, \xi),\:b(t, x,\xi)$ are real-valued symbols  homogeneous of degree 0 in $\xi$  satisfying
$$
\tilde{a}(t, x,\xi) \geq c_0 > 0 \;\; {\rm for}\;\;(t, x, \xi) \in [0, T] \times T^*(U),
$$
$$\alpha(x, \xi) \geq 0,\;\;(x,\xi) \in T^*(U).
$$
 Set
$a(t, x, \xi) = (t + \alpha)\tilde{a}(t, x, \xi).$ Since $P$ is hyperbolic for $ t \geq 0$, the symbol $\Delta$  satisfies the inequality
\begin{equation} \label{eq:1.5}
\Delta(t, x,\xi):  = \Bigl(4 a^3(t, x, \xi)- 27 b^2(t, x, \xi)\Bigr) \geq 0, \;\; (t, x,\xi) \in [0, T] \times T^*(U)
\end{equation}
and this guarantees that the equation $p = 0$ with respect to $\tau$ has only real roots for $(t, x,\xi) \in [0, T] \times T^*(U)$. The set $\Sigma_3$ has the form
$$\Sigma_3 = \{(t, x, \xi):\: t = \tau = 0, \: \alpha(x, \xi) = 0\}.$$
We can write the symbol $b(t, x,\xi)$ as
$$b(t, x, \xi) = \beta_0(x, \xi) + t \beta_1(x, \xi)+ t^2 \beta_2(t, x, \xi).$$
The hyperbolicity of $P$ implies
$$\beta_0\big\vert_{\Sigma_3} =\beta_1\big\vert_{\Sigma_3} = 0.$$
Under the stronger assumption that $\beta_0(x,\xi)$ and $\beta_1(x, \xi)$ vanish in a small neighborhood of every point of $\Sigma_3$ it has been proved in \cite{BBP} that $P$ is strongly hyperbolic.\\

In this paper we prove that $P$ is strongly hyperbolic under weaker assumption (E) (see section 2) saying that microlocally we have
\begin{equation} \label{eq:1.6}
\Delta(t, x, \xi) \geq \delta t ( t + \alpha(x, \xi))^2, \;\; \delta > 0.
\end{equation}
According to Lemma 2.1, this condition is satisfied under an estimate on $\beta_1$. This generalizes the result in \cite{BBP}. 
Since the symbols $\Delta$ and $\Delta_0$ are invariant under the change of variables preserving $t$, it follows that (E) is equivalent to the condition
$$\Delta(t, x, \xi) \geq \delta_1 t\Delta_0(t, x, \xi),\;\; \delta_1 > 0,\;\;|\xi|= 1.$$

To understand the condition (\ref{eq:1.6}), consider the case of a hyperbolic operator $P$ with double characteristic with principal symbol
$$p_2 =\tau^2 - q(t, x,\xi),$$
where $q(t, x, \xi) \geq  0$. The first author proved in \cite{N2} that $P$ is effectively hyperbolic if and only if for every double characteristic point $z_0=(t_0,x_0,0,\xi_0)$ where $q(t_0,x_0,\xi_0)=0$
 there exists a smooth   function $f(t, x,\xi) $ (independent of $\tau$) which is  homogeneous of degree 0 in $\xi$ and defined in  a small conic neighborhood $W$ of $z'_0 = (t_0, x_0, \xi_0)$ such that  
$$
q(t, x, \xi) \geq \delta f^2(t, x, \xi) |\xi|^2, \;\; \delta > 0,\;\;(t, x, \xi) \in W
$$
and 
\begin{equation}
\label{eq:timef}
C_{z_0} \cap T_{z_0}\{(t, x, \tau, \xi): f(t,x,\xi) = 0\} =\{0\},\:\text{(equivalently  $-H_f(z_0) \in C_{z_0}$)},
\end{equation}
$C_{z_0}$ being the {\it propagation cone} at $z_0$. Such $f(t,x,\tau, \xi)$ verifying \eqref{eq:timef} is called a time function at $z_0$. Notice that if $z_j$ are simple characteristic points and if 
$$z_j \to z_0,\;\;
\lambda_j H_{p_2}(z_j) \to X,\;\;\lambda_j >0,
$$
 then $X \in C_{z_0}$. So $C_{z_0}$ contains the directions of all (simple) bicharacteristics converging to $z_0$.
We have a loss of regularity passing across the manifold given by $f= 0$ and for this reason it is convenient to obtain
microlocal energy estimates with weight $f^{-N}$ and a big $N \gg 1$ depending on lower order terms.

The situation on the case  with triple characteristics points  is quite different. We show in Example 2.2 in the next section that
there are effectively hyperbolic operators of third order with triple characteristics points lying on $t = 0$ for which even a weaker condition 
$$
\Delta \geq \delta t^m (t + \alpha(x, \xi))^k ,\;\; \delta > 0
$$
with arbitrary large $m, k \in \N$ is not microlocally satisfied. Thus in general the construction of a time function seems impossible and it is an open problem if the operator in Examples 2.2 is strongly hyperbolic.\\

Our approach is based on two new ideas. First one reduces the equation $P u = f$ to a first order pseudo-differential system
$D_t U(t) =A\la D \ra U(t) +B U(t) + F$ with 
\[
 A=\left(\begin{array}{ccc}
0&a&b\\
1&0&0\\
0&1&0\end{array}\right)
\]
which is not diagonalizable for $z \in \Sigma_3$. Hence it is not possible to construct a  positive definite symmetrizer for $A$.
On the other hand, it is easy to see that a  non-negative definite symmetrizer $S(t)$ exists. To obtain energy estimates with (big) loss of derivatives, we need to work with an energy ${\mathsf{Re}}\,(S(t)U, U)_{L^2(\R^n)}$ which in our case is not positive. An application of the sharp
G\aa rding inequality for ${\mathsf{Re}}\,(S(t)U, U)_{L^2(\R^n)}$ is not sufficient, while the Fefferman-Phong inequality for matrix symbols
$$
{\mathsf{Re}}\,(S(t)U, U)_{L^2(\R^n)} \geq -C \|\la D \ra^{-1} U\|^2,\;\; 0 \leq t \leq T,\;\;C >0
$$
in general is false (see \cite{Bru}, \cite{Pa}). Our second idea is to prove under the condition (E) the following inequality
\[
{\mathsf{Re}}(S(t)U,U)_{L^2(\R^n)}\geq \delta t\big(\sum_{j=1}^2\|U_j\|^2+(aU_3,U_3)\big)-Ct^{-1}\|\lr{D}^{-1}U\|^2, \; 0 < t \leq T,\: \delta > 0,\: C > 0
\]
which is a version of Fefferman-Phong inequality with a constant $C t^{-1}$.
By using the latter inequality, we define the energy $(\tilde{S}(t)U, U)$ with
$\tilde{S}(t) = S(t) + \lambda t^{-1} \la \xi \ra^{-2}I$ choosing $\lambda \geq \lambda_0$ large enough. The main point is to obtain local energy estimates with weight $t^{-N}$ (see Theorem \ref{pro:energy}). First, in Section 3 we obtain these estimates assuming (E) being globally satisfied. Next in Section 4  we show how we can deduce  global estimates from  microlocal ones. Consequently, under the assumptions of Theorem \ref{pro:energy} the operator $P$ is strongly hyperbolic. Finally, in Appendix
we give a representation of a Friedrichs symmetrization by  Weyl quantized pseudo-differential operator.
 
\section{Hyperbolic symbols with triple characteristics}

In this section we use the notation of Section 1 and for brevity we will write below $\Delta,\: \alpha, \tilde{a}$ instead of $\Delta(t, x, \xi),\:\alpha(x, \xi), \:\tilde{a}(t, x,\xi)$.
Assume that $\alpha(x_0, \xi_0) = 0.$ For our analysis of the Cauchy problem we introduce two conditions:\\

(H) There exist a small conic neighborhood $W$ of $(x_0, \xi_0)$ and a small $\delta > 0$ such that for $(x, \xi) \in W$ and small $t \geq 0$ we have
\begin{equation} \label{eq:2.9}
\Delta(t, x, \xi) \geq \delta t^2(t + \alpha(x, \xi)).
\end{equation}

(E) There exist a small conic neighborhood $W$ of $(x_0, \xi_0)$ and a small $\delta > 0$ such that for $(x, \xi) \in W$ and small $t \geq 0$ we have
\begin{equation} \label{eq:2.10}
\Delta(t, x, \xi) \geq \delta t(t + \alpha(x, \xi))^2.
\end{equation}
Clearly, (E) implies (H), but the inverse is not true. It is easy to see that (H) is satisfied if there exists $\epsilon > 0$ such that
$$
\Delta(t, x,\xi) \geq \epsilon (t^3 + \alpha^3),\;\; (x,\xi) \in W.
$$
\begin{lem} Assume that for $(x, \xi)\in W$ we have
\begin{equation} \label{eq:2.11}
|\beta_1(x,\xi)|\leq \frac{1}{\sqrt{3}}\sqrt{\alpha(x,\xi)}.
\end{equation}
Then $(H)$ holds for $(x, \xi)$ in a possibly smaller neighborhood $W_1 \subset W$. Next assume that there exists
$\epsilon > 0$ such that for $(x, \xi) \in W$ we have
\begin{equation} \label{eq:2.12}
|\beta_1(x,\xi)|\leq \frac{1 - \epsilon}{\sqrt{3}}\sqrt{\alpha(x,\xi)}. 
\end{equation}
 Then $(E)$ holds for $(x, \xi) \in W_1$.
\end{lem}
 \begin{proof} 
 Assume that \eqref{eq:2.11} is fulfilled. Since $|\beta_0| \leq \frac{2}{3 \sqrt{3}} \alpha^{3/2}$, we get
$$2 \alpha^2 - 9 \beta_1 \beta_0 \geq 2 \alpha^{3/2}[ \alpha^{1/2} - \sqrt{3} |\beta_1|] \geq 0.$$
 Then $12\alpha^2 - 54 \beta_1 \beta_0 \geq 0$ and we have
$$\Delta \geq \Bigl(4 - 54 \beta_2 \beta_1 - 27 t \beta_2^2\Bigr)t^3  + (12\alpha - 27 (\beta_1^2 + 2 \beta_2 \beta_0))t^2.$$
Next, for small $\alpha$ and $t > 0$ for $(x, \xi) \in W_1$ we arrange
$$12\alpha - 27 (\beta_1^2 + 2 \beta_2 \beta_0) \geq 3(4 \alpha - 3 \alpha - 4 \sqrt{3}|\beta_2| \alpha^{3/2}) \geq 2 \alpha,$$
$$4 - 54 \beta_2 \beta_1 - 27 t \beta_2^2 \geq 2,$$
and this implies
$$\Delta \geq 2t^2(t + \alpha).$$
Now suppose that \eqref{eq:2.12} is satisfied. Then $9|\beta_1 \beta_0| \leq 2(1 -\epsilon) \alpha^2$ and
$12\alpha^2 - 54 \beta_1 \beta_0 \geq \epsilon \alpha^2$. This implies for small $t$ and $\alpha$ the estimate
$$ \Delta \geq 2 t^3 + \epsilon t\alpha^2 \geq \epsilon_1 t (t + \alpha)^2,\;\; \epsilon_1 > 0$$
and the proof is complete.
\end{proof}

We present below some examples when the condition (H) is not satisfied.\\

{\bf Example 2.1.} Consider the symbols
\begin{equation} \label{eq:1}
p_{\pm}(t, x, \tau, \xi) = \tau^3 -(t + x^2)\tau\xi^2 \pm  t x\xi^3.
\end{equation}

Clearly,  $p_{\pm}$ are hyperbolic for $ t \geq 0$, since
$$\Delta_{\pm} = 4 (t + x^2)^3 - 27t^2 x^2  = (t - 2 x^2)^2 (4t + x^2)\geq 0,\;\; t \geq 0. $$
For $ t = 2 x^2$ and $x \neq 0$ we have a double characteristic roots, while for $t = x = 0$ we have a triple root.
For $t = x = 0$ the symbols $p_{\pm}$ are effectively hyperbolic and for small $t \geq 0$ and small $x$ the fundamental matrix $F_{p_{\pm}}$ has
non-zero real eigenvalues by perturbation. On the other hand, both symbols admit a factorization with a smooth root $\tau = \pm x \xi$, that is
$$
p_{+} = (\tau^2 + x\tau \xi - t \xi^2)(\tau - x \xi),
$$
$$
p_{-} = (\tau^2 - x \tau \xi - t \xi^2)(\tau + x \xi).
$$

{\bf Examples 2.2.}
Consider the symbol
\begin{equation} \label{eq:1}
p(t, x, \tau, \xi) = \tau^3 -(t + \alpha(x))\xi^2\tau  + (t^{m}/2 - t )\sqrt{\alpha(x)}\xi^3,\;\; m \gg 2,
\end{equation}
where $\alpha(x) \geq 0$ and $\sqrt{\alpha(x)}$ is smooth.
 The symbol $p$ is hyperbolic for $0 \leq  t \leq 1$ since
$$\Delta = 4 (t + \alpha)^3 - 27 (t^m/2 - t)^2 \alpha$$
$$= 4t^3 - 15 t^2 \alpha +12t \alpha^2 + 4 \alpha^3 + 27 t^{m+1}\alpha(1 - t^{m - 1}/4)$$
$$= ( t - 2\alpha)^2 (4 t + \alpha) + 27 t^{m+1} \alpha( 1-  t^{m-1}/4) \geq 0$$
and $\Delta >0$ for $t +\alpha > 0.$

For $ t = 2 \alpha$ we have $\Delta = 27\cdot 2^{m+1} \alpha^{m+2}(1 -  2^{m-3}\alpha^{m-1})$ and clearly we cannot have an estimate
$$\Delta \geq \delta t^k( t + \alpha)^q,\;\; \delta > 0 $$
for small $ \alpha> 0$ if $m > k + q - 2.$

Finally, we will discuss some link between the conditions (H) and (E) and the behavior of the real roots $\lambda_k(t, x, \xi),\: k =1, 2,3$ of the equation $p = 0$ with respect to $\tau.$
These roots have the trigonometric form
$$\begin{cases} \lambda_1 = 2\rho\cos(\theta/3),\\
\lambda_2 = 2\rho \cos(\theta/3 + \frac{2\pi}{3}),\\
\lambda_3 =  2\rho\cos(\theta/3 + \frac{4 \pi}{3}),\end{cases}$$
where 
$$\rho = \Bigl(\frac{a}{3}\Bigr)^{1/2}\la \xi \ra,\;\; \theta = \arccos \Bigl(\frac {3 \sqrt{3} b}{2a^{3/2}}\Bigr).$$
We assume below that $|\xi|= 1.$ If the condition (H) holds, for $t \geq 0$ we have
$$\sin^2\theta =1 - \cos^2 \theta = \frac{ \Delta}{4(t+ \alpha)^{3}\tilde{a}^{3}} \geq c_0^2 \frac{t^2}{(t + \alpha)^2}$$
hence $|\sin\theta| \geq c_0 t (t + \alpha)^{-1}, \: c_0 >0$ if $t + \alpha > 0.$

We apply the elementary inequality
$$|\sin \theta| = |\sin (\theta /3)| ( 3 \cos^2 (\theta/3)- \sin^2(\theta/3))| \leq 4|\sin(\theta/3)|$$
and consider two cases:  (i) $-\pi/2\leq \theta \leq \pi/2$, (ii) $\pi/2 < \theta < 3\pi/2.$ 
First we deal with the case (i). 
We obtain 
$$
|\sin(\theta/3)| \geq |\sin \theta|/4 \geq \frac{t}{4 c_0(t + \alpha)}.
$$
Next $$|\lambda_2 - \lambda_3| = 2 \rho \Bigl|\cos\Bigl( \theta/3 + \frac{2\pi}{3}\Bigr) -\cos \Bigl(\theta/3 + \frac{4\pi}{3}\Bigr)\Bigr|
$$
$$ =  4 \rho |\sin (\theta/3)| \sin (\frac{2 \pi}{3}) \geq C_0 \frac{t}{\sqrt{t + \alpha}}.
$$
For the differences $\lambda_1 - \lambda_2,\: \lambda_1 - \lambda_3$ it is easy to obtain bounds, since we have obvious
estimates for  $\sin(\theta/3 + \pi/3)$ and $\sin(\theta/3 + 2 \pi/3)$ because
$$
\pi/6 \leq \theta/3 +\pi/3\leq \pi/2,\;\;   \pi/2 \leq\theta/3 + 2 \pi/3\leq \frac{5}{6}\pi.
$$
Thus we have
\[
C_1\sqrt{t}\leq C_1 \sqrt{(t + \alpha)}\leq |\lambda_1 - \lambda_k| \leq C_2 \sqrt{(t + \alpha)},\;\; k = 2, 3.
\]
Consequently,
\begin{equation}
|(\lambda_1 - \lambda_k) (\lambda_2 - \lambda_3)| \geq C_0C_1  t, \;\; k = 2, 3,
\end{equation}
while 
\begin{equation}
 |(\lambda_1 - \lambda_2)(\lambda_1 - \lambda_3)| \geq C_1^2 (t + \alpha) \geq C_1^2 t.
\end{equation}
The analysis of the case (ii) is very similar. In this case we have
$$
\frac{\pi}{6} < \theta/3 < \frac{\pi}{2}, \;\; \frac{\pi}{2} < \theta/3  + \frac{\pi}{3}< \frac{5 \pi}{6},
$$
while
$$ \frac{5 \pi}{6} < \theta/3 + \frac{2 \pi}{3} < \frac{7\pi}{6}.$$
To estimate $|\sin(\theta/3 + 2 \pi/3)|$, notice that $3 (\theta/3 + 2 \pi/3) = \theta + 2 \pi,$
hence 
$$
|\sin\Bigl(\theta/3 + \frac{2\pi}{3}\Bigr)| \geq \frac{|\sin \theta|}{4}
$$
 and we repeat the above argument.\\

In the case when (E) holds one has a sharper result, since for $t + \alpha > 0$ we have
 $$|\sin\theta| \geq c_0\sqrt{\frac{t}{t + \alpha}}, \;\; c_0 >0$$
and 
$$|\lambda_2 - \lambda_3| \geq C_0 \sqrt{t},
$$
$$|(\lambda_1 -\lambda_k)(\lambda_2 -\lambda_3) | \geq C_2 \sqrt{t(t + \alpha)},\: k = 2, 3.$$
Thus we have proved the following
\begin{prop} Let
$$\delta_k(t, x, \xi) = \frac{\pa p(t, x, \tau, \xi)}{\pa \tau}\Bigl\vert_{\tau = \lambda_k(t, x, \xi)},\;\; k = 1,2,3$$
and assume the condition $(H)$. Then there exist a constant $c_1 > 0$ and a conic neighborhood $W_1$ of $(x_0, \xi_0)$ such that for sufficiently small $t \geq 0$ and $(x,\xi) \in W_1$ we have
\begin{equation}
|\delta_k(t, x,\xi) | \geq c_1 t \la \xi \ra^2, \;\; k = 1, 2, 3.
\end{equation}
In the case when the condition $(E)$ is satisfied we have a sharper estimate
\begin{equation}
|\delta_k(t, x, \xi)| \geq c_1 \sqrt{t(t + \alpha)} \la \xi \ra ^2,\:\: k = 1,2,3.
\end{equation}
\end{prop}

\section{Symmetrizer and energy estimates}

In this paper we work with symbols $a(t,x,\xi)$ which depend on $t$ smoothly  and use the Weyl quantization of $a$
\[
a(t,x,D)u=({\rm Op}^w(a)u)(x)=(2\pi)^{-n}\int\int e^{i(x-y)\xi }a\Bigl(t, \frac{x+y}{2},\xi\Bigr)u(y)dyd\xi.
\]
We will often write $au$ instead of $a(t,x,D)u$ or ${\rm Op}^w(a)u$ if there is no confusion. We abbreviate  $S^m_{1,0}$ to $S^m$.
%%%%%%%%%%%
%%
%%%%%%%%%%%
We study the operator  $P$
\begin{equation}
\label{eq:takata}
P=D_t^3-a(t,x,D)D_t\lr{D}^2-b(t,x,D)\lr{D}^3+\sum_{j=0}^2b_{1j}(t,x,D)D_t^{2-j}\lr{D}^j
\end{equation}
which is differential operator in $t$ where $b_{1j}\in S^0$. Here we assume  
\begin{equation}
\label{eq:tataki}
a(t,x,\xi)=(t+\al(x,\xi)){\tilde a}(t,x,\xi),
\end{equation}
where $0<c_0\leq {\tilde a}(t,x,\xi)\in S^0$ and $b(t,x,\xi)\in S^0$ depend smoothly on $t\in [0,T]$ and $0\leq \al(x,\xi)\in S^0$, 
\begin{equation}
\label{eq:1.5bis}
\Delta(t, x,\xi) = 4 a^3(t, x, \xi)- 27 b^2(t, x, \xi) \geq 0, \quad (t, x,\xi) \in [0, T]\times \R^n\times\R^n. 
\end{equation}
%
%%%%%%
We will study the Cauchy problem for $P$ for $t \geq 0$. 
  According to Lemma 8.1 in \cite{IP}, we have
\begin{equation} \label{eq:2.3}
\grad_x\alpha(x_0,\xi_0) = 0,\;\;\grad_{t, x, \xi}b(0, x_0, \xi_0) = 0,
\end{equation}
\begin{equation} \label{eq:2.4}
\pa^2_{x,\xi} b(0,x_0,\xi_0) = \pa^2_{x, x}b(0,x_0, \xi_0) = 0
\end{equation}
if $\alpha(x_0,\xi_0)=0$. The hyperbolicity \eqref{eq:1.5bis} implies stronger conditions: 
\begin{lem} 
\label{lem:sou} For small $t \geq 0$ one has
\begin{equation} \label{eq:2.5}
\pa_t b(t, x, \xi) = \co(\sqrt{a}),\;\; \pa^{\beta}_{x}\pa^{\gamma}_{\xi} b(t, x, \xi) = \co(a),\;\;|\beta + \gamma| = 1,
\end{equation}
\begin{equation}\label{eq:2.6}
\pa^{\beta}_x \pa^{\gamma}_{\xi} b(t, x, \xi)=\co (\sqrt{a}),\;\; |\beta + \gamma| = 2.
\end{equation}
\end{lem}
\begin{proof} Set $X = (y, \eta),\: |X| \leq 1$ and consider \eqref{eq:1.5bis} which implies  for $|s| \leq s_0$ with some $s_0>0$
\begin{equation}
 \label{eq:2.7}
 \begin{split}
|b(t, x +s y, \xi + s \eta)| =|b(t,x, \xi) + s\la\nabla b, X \ra +s^2 \la H(b)X, X\ra + \co(s^3)|\\
\leq C\Bigl(a(t, x, \xi) + s \la \nabla a, X \ra + s^2 \la H(a) X, X \ra + \co(s^3)\Bigr)^{3/2},
\end{split}
\end{equation}
where $H(b), \: H(a)$ denote the Hessian of $b$ and $a$ with respect to $(x,\xi)$ and $\nabla b, \:\nabla a$ is the gradient with respect to $(x, \xi)$. If $a(t, x, \xi) = 0$, we deduce
that $b(t, x, \xi) =0$, while $a(t, x, \xi) \geq 0$ implies $\nabla a(t, x, \xi) = 0$. Therefore
$$
|s \la \nabla b, X \ra + s^2 \la H(b)X, X \ra| \leq C |s|^3
$$
and we get $\nabla b(t, x, \xi) = 0$ and $H(b)(t, x, \xi) = 0.$ Now assume that $a(t, x, \xi) \neq 0.$ If $a(t, x,\xi) \geq s_0^2,$ then
$$
|\nabla b| \leq C = C s_0^{-2} s_0^2 \leq C_1s_0^{-2} a,\;\; \|H(b)\| \leq C = C s_0^{-1} s_0 \leq C_2 s_0^{-1} \sqrt{a}.
$$
Now suppose that $0 < a(t, x, \xi) < s_0.$ Then we take $s = \sqrt{a(t, x,\xi)}$ in (\ref{eq:2.7}) and obtain
\begin{align*}
 &|b +\sqrt{a} \la \nabla b, X \ra + a \la H(b)X, X \ra + \co(a^{3/2})|\\
 &\leq C \Bigl( a + \sqrt{a}\la \nabla a, X \ra + a\la H(a) X, X \ra + \co (a^{3/2})\Bigr)^{3/2}\leq C a^{3/2}
 \end{align*}
since $|\nabla a|\leq C \sqrt{a}$ by Gleaser inequality. Consequently,
$$
\Bigl|\frac{b}{a^{3/2}}+\frac{1}{a}\la \nabla b, X \ra + \frac{1}{\sqrt{a}} \la H(b) X, X \ra + \co(1)\Bigr| \leq C.
$$
Since $|b/a^{3/2}| \leq C$, we replace $X$ by $\mu X$ with $|X| = 1$ and $0 < |\mu| \leq 1$ and deduce
$$
|\mu| \Bigl| \frac{1}{a} \la \nabla b, X \ra + \mu \frac{1}{\sqrt{a}} \la H(b) X, X \ra \Bigr| \leq C.
$$
Choosing $\mu = \pm 1$, one gets
$$
- C \leq \frac{1}{a}\la \nabla b, X \ra \pm \frac{1}{\sqrt{a}}\la H(b) X, X \ra \leq C
$$
and this yields
$$
\Bigl| \la \nabla b, X \ra \Bigr| \leq C a,\;\; \Bigl|\la H(b) X, X \ra \Bigr| \leq C \sqrt{a}.
$$
Since $X$ with $|X| = 1$ is arbitrary, we obtain the desired estimates. 

For the derivative $\pa_t b$ we apply (\ref{eq:2.3}) if $t + \alpha =0$. To examine  $\pa_t b$ for $t + \alpha \neq 0$, set
$$
\beta_0(x,\xi) = b(0,x,\xi),\;\;\beta_1(x,\xi) = (\pa_t b)(0,x,\xi)
$$
and write
$$
b(t, x, \xi) = \beta_0(x, \xi) + t \beta_1(x, \xi)+ t^2 \beta_2(t, x, \xi).
$$
It is sufficient to prove that
\begin{equation} \label{eq:2.8}
\beta_1(x,\xi) = \co(\sqrt{\alpha}).
\end{equation}
From $\Delta \geq 0$ we deduce
$$
(t^2 \beta_1 + t \beta_1 + \beta_0)^2 \leq C (t + \alpha)^3.
$$
If $\alpha(x, \xi) = 0$, we obtain $\beta_0(x, \xi) = 0$ and hence $\beta_1(x, \xi) = 0$ because $t > 0.$ If $\alpha(x, \xi) \neq 0$, we take $t = \alpha(x, \xi)$ and since $\beta_0 = \co(\alpha^{3/2})$, one concludes that
$$
\alpha |\beta_1| = \co(\alpha^{3/2})
$$
which yields (\ref{eq:2.8}). This completes the proof.
\end{proof}

%%%%%%
 In the following up to the end of this section we assume that \eqref{eq:2.10} is satisfied globally, that is
 \begin{equation}
 \label{eq:jyokenE}
 \Delta(t, x,\xi)  \geq \delta t(t+\alpha(x,\xi))^2, \quad (t, x,\xi) \in [0, T]\times \R^n\times\R^n. 
 \end{equation}
With $U={^t}(D_t^2u,D_t\lr{D}u,\lr{D}^2u)$ the equation $P u=f$ is reduced to 
\begin{equation}
\label{eq:redE}
D_tU=(A\lr{D}+B)U+F,
\end{equation}
where $F={^t}(f,0,0)$ and
\[
 A=\left(\begin{array}{ccc}
0&a&b\\
1&0&0\\
0&1&0\end{array}\right),\quad B= \left(\begin{array}{ccc}
b_{10}&b_{11}&b_{12}\\
0&0&0\\
0&0&0\end{array}\right).
\]
 Denote
\[
S(t,x,\xi)=\left(\begin{array}{ccc}
3&0&-a\\
0&2a&3b\\
-a&3b&a^2\end{array}\right)
\]
which is a symmetrizer of $A$ such that   
\[
S(t,x,\xi)A(t,x,\xi)=\left(\begin{array}{ccc}
0&2a&3b\\
2a&3b&0\\
3b&0&-ab\end{array}\right)
\]
is symmetric. 
\begin{lem}
\label{lem:G2} There exists $\delta>0$ such that
\[
S(t,x,\xi)\gg 2\delta t\left(\begin{array}{ccc}
1&0&0\\
0&1&0\\
0&0&a\end{array}\right)=2\delta t J.
\]
\end{lem}
\begin{proof}
  Since 
\[
S-2\delta t J=\left(\begin{array}{ccc}
3-2\delta t&0&-a\\
0&2a-2\delta t&3b\\
-a&3b&a^2-2\delta ta\end{array}\right),
\]
it is enough to note
\[
{\rm det}\,(S-2\delta tJ)=\Delta+2\delta \co \big(t(t+\al)^2\big).
\]
\end{proof}
To  derive energy estimates we apply the sharp G\aa rding inequality for which proof we employ the symmetrization method of Friedrichs and we make detailed looks at the difference between the original operator and its Friedrichs symmetrization as in \cite{N3}.  The details of this representation are given in the Appendix for convenience sake.
Denote $Q=S-\delta tJ$. Then $Q\in S^0$ and $Q\gg 0$, that is $Q$ is non-negative definite by Lemma \ref{lem:G2}. Denote by  $Q_F$ the Friedrichs part of $Q$ (the Friedrichs symmetrization of $Q$) which we define
precisely in the Appendix. By the construction of $Q_F$ we get  $(Q_FU,U)\geq 0$ for any $U\in C^{\infty}(\R_t: C_0^{\infty}(\R^n))$. For our argument we need a more precise representation of the
difference 
\begin{equation}
\label{eq:sGar}
\begin{split}
Q_F-{\rm Op}^w(Q)={\rm Op}^w\Big(
\sum_{2\leq |\al+\be|\leq 3}\psi_{\al,\be}(\xi)Q^{(\al)}_{(\be)}\Big)+{\rm Op}^w(R),\quad R\in S^{-2}_{1/2,0},
\end{split}
\end{equation}
where $\psi_{\al,\be}\in S^{(|\al|-|\be|)/2}$  are real symbols and we have used the notation $Q^{(\al)}_{(\be)}=\dif_{\xi}^{\al}D_x^{\be}Q$. We give a proof of \eqref{eq:sGar} in Appendix.
\begin{lem}
\label{lem:kagi} There exists $C>0$ such that for $U \in C^{\infty}(\R_t:C_0^{\infty}(\R^n))$ we have
\[
{\mathsf{Re}}(SU,U)\geq \delta t\big(\sum_{j=1}^2\|U_j\|^2+(aU_3,U_3)\big)-Ct^{-1}\|\lr{D}^{-1}U\|^2.
\]
\end{lem}
\begin{rem} It is important to note that the sharp G\aa rding inequality for the matrix operators implies the
estimate
\[
{\mathsf{Re}}\,(S U,U)\geq -C\|\lr{D}^{-1/2}U\|^2.
\]
On the other hand, in general the Fefferman-Phong type inequality for matrix operators with non-negative symbols does not hold (see \cite{Bru}, \cite{Pa}).

\end{rem}
\begin{proof} Notice that $a$ is real, hence $(a U_3, U_3) = {\mathsf{Re}}\, (a U_3, U_3).$ Since  
$${\mathsf{Re}}\,(SU,U)={\mathsf{Re}}\,({\rm Op}^w(Q)U,U)+2\delta t\big(\sum_{j=1}^2\|U_j\|^2+(aU_3,U_3)\big),$$
 it is enough to prove
\begin{equation}
\label{eq:katuo}
\big|{\mathsf{Re}}({\rm Op}^w\Bigl(\sum_{2 \leq |\al + \be| \leq 3}\psi_{\al,\be}Q_{(\be)}^{(\al)}\Bigr)U,U)\big|\leq \delta t\big(\sum_{j=1}^2\|U_j\|^2+(aU_3,U_3)\big)+C\delta^{-1}t^{-1}\|\lr{D}^{-1}U\|^2.
\end{equation}
Indeed if this is true, then we have 
\begin{align*}
{\mathsf{Re}}({\rm Op}^w(Q)U,U)\geq (Q_FU,U)-\delta t\big(\sum_{j=1}^2\|U_j\|^2+(aU_3,U_3)\big)\\-C\delta^{-1}t^{-1}\|\lr{D}^{-1}U\|^2
-C\|\lr{D}^{-1}U\|^2\\
\geq -\delta t\big(\sum_{j=1}^2\|U_j\|^2+(aU_3,U_3)\big)-C\delta^{-1}t^{-1}\|\lr{D}^{-1}U\|^2,
\end{align*}
hence we conclude the assertion.
 
To prove \eqref{eq:katuo}, consider ${\mathsf{Re}}({\rm Op}^w(\psi_{\al,\be}Q_{(\be)}^{(\al)})U,U)$ with $|\al+\be|=3$. Note that
\[
Q^{(\al)}_{(\be)}=\left(\begin{array}{ccc}
0&0&-a^{(\al)}_{(\be)}\\
0&2a^{(\al)}_{(\be)}&3b^{(\al)}_{(\be)}\\
-a^{(\al)}_{(\be)}&3b^{(\al)}_{(\be)}&(a^2)^{(\al)}_{(\be)}-\delta ta^{(\al)}_{(\be)}\end{array}\right).
\]
Since $\psi_{\al,\be}a^{(\al)}_{(\be)}, \psi_{\al,\be}b^{(\al)}_{(\be)}\in S^{-3/2}$,   it is easy to see that
\begin{equation}
\label{eq:komazo}
\begin{split}
{\mathsf{Re}}({\rm Op}^w(\psi_{\al,\be}Q_{(\be)}^{(\al)})U,U)\leq \delta t\sum_{j=1}^2\|U_j\|^2+C\delta^{-1}t^{-1}\|\lr{D}^{-3/2}U\|^2\\
+{\mathsf{Re}}({\rm Op}^w(\psi_{\al,\be}((a^2)_{(\be)}^{(\al)}-\delta ta_{(\be)}^{(\al)}))U_3,U_3).
\end{split}
\end{equation}
To estimate the third term on the right-hand side, one can assume that $\psi_{\al,\be}((a^2)_{(\be)}^{(\al)}-\delta ta_{(\be)}^{(\al)})\in S^{-3/2}$ is real. Writing $\psi_{\al,\be}((a^2)_{(\be)}^{(\al)}-\delta ta_{(\be)}^{(\al)})={\mathsf{Re}}\,(T\#\lr{\xi}^{-3/2})+R$ with $T=\lr{\xi}^{3/2}\psi_{\al,\be}\big((a^2)_{(\be)}^{(\al)}-\delta ta_{(\be)}^{(\al)}\big)$ and $R\in S^{-2}$,  one gets
\begin{equation}
\label{eq:komazo:2}
{\mathsf{Re}}({\rm Op}^w(\psi_{\al,\be}((a^2)_{(\be)}^{(\al)})-\delta ta_{(\be)}^{(\al)})U_3,U_3)\leq \delta t\|TU_3\|^2+C\delta^{-1}t^{-1}\|\lr{D}^{-1}U_3\|^2.
\end{equation}
Thanks to Glaeser inequality we have $|(a^2)^{(\al)}_{(\be)}|\leq C'\sqrt{a}$ which yields $C\sqrt{a}-T\geq 0$  with some $C>0$ because $a\geq t$. Note that $\|TU_3\|^2=({\rm Op}^w(T\#T)U_3,U_3)$ and ${\mathsf{Re}}\, (T\#T) -T^2\in S^{-2}$. On the other hand, since $S^0\ni  Ca-T^2\geq 0 $, the Fefferman-Phong inequality for scalar symbols proves that
\begin{equation}
\label{eq:FeP}
C(aU_3,U_3)\geq \|TU_3\|^2-C\|\lr{D}^{-1}U_3\|^2
\end{equation}
from which we have
\begin{equation}
\label{eq:hana}
{\mathsf{Re}}({\rm Op}^w(\psi_{\al,\be}((a^2)_{(\be)}^{(\al)}-\delta ta_{(\be)}^{(\al)}))U_3,U_3)\leq \delta t(aU_3,U_3)+C\delta^{-1}t^{-1}\|\lr{D}^{-1}U_3\|^2.
\end{equation}
For the case $|\al+\be|=2$, observing that $\psi_{\al,\be}(a^2)^{(\al)}_{(\be)}\in S^{-1}$,  it suffices to repeat the same arguments.
\end{proof}
\begin{cor}
\label{cor:seiti}Let ${\tilde S}=S+\lambda\, t^{-1}\lr{\xi}^{-2}I$. Then there exists $\lambda_0>0$  such that for $\lambda \geq \lambda_0$ we have
\[
\begin{aligned}
{\mathsf{Re}}({\tilde S}U,U)={\mathsf{Re}}(SU,U)+\lambda t^{-1}\|\lr{D}^{-1}U\|^2\\
\geq \delta t\big(\sum_{j=1}^2\|U_j\|^2+(aU_3,U_3)\big)+(\lambda/2) t^{-1}\|\lr{D}^{-1}U\|^2.
\end{aligned}
\]

\end{cor}
\begin{cor}
\label{cor:seiti:2}There exist $\delta_1>0$ and $\lambda_0>0$  such that 
\[
{\mathsf{Re}}({\tilde S}U,U)\geq \delta_1 t^2\|U\|^2+(\lambda/2) t^{-1}\|\lr{D}^{-1}U\|^2,\quad \lambda\geq \lambda_0.
\]
\end{cor}
\begin{proof}Clearly,  there exists $\delta_1>0$ such that $a\geq \delta_1 t$ by \eqref{eq:tataki}. Then from the Fefferman-Phong inequality for the symbol $a - \delta_1 t$ one has 
\[
(aU_3,U_3)\geq \delta_1 t\|U_3\|^2-C\|\lr{D}^{-1}U_3\|^2
\]
which proves the assertion thanks to Corollary \ref{cor:seiti}.
\end{proof}
%
%%%%%%%%%%%%%%%%%%%%%%%%%
%%%%%%%%%%%%%%%%%%%%%%%%%
Consider the energy $(t^{-N}e^{-\gam t}{\tilde S}U,U)$, where $(\cdot,\cdot)$ is the $L^2(\R^n)$ inner product and $N>0$, $\gam>0$ are positive parameters. Then one has
\begin{equation}
\label{eq:kiso}
\begin{split}
\dif_t(t^{-N}e^{-\gam t}{\tilde S}U,U)=-N(t^{-N-1}e^{-\gam t}{\tilde S}U,U)-\gam (t^{-N}e^{-\gam t}{\tilde S}U,U)\\
+ (t^{-N}e^{-\gam t}\dif_t SU,U)
-\lambda (N+1) t^{-N-2}e^{-\gam t}\|\lr{D}^{-1}U\|^2 - \lambda \gam t^{-N -1} e^{-\gam t}\|\lr{D}^{-1} U\|^2\\
-2{\mathsf{Im}}\,(t^{-N}e^{-\gam t}(({\tilde S}(A\lr{D}+B)U,U)-2{\mathsf{Im}}(t^{-N}e^{-\gam t}{\tilde S}F,U).
\end{split}
\end{equation}
Consider $2{\mathsf{Im}}(SA\lr{D}U,U)=-i\big((SA\lr{D}-\lr{D}A^*S)U,U\big)$. By the calculus of Weyl pseudo-differential operators for the symbol of the
operator $SA\lr{D} - \lr{D}A^*S$ we get the representation
\[
\begin{aligned}
&S\#A\#\lr{\xi}-\lr{\xi}\#A^*\#S\\
&= \sum_{1\leq |\al+\be+\delta+\nu|\leq 2}\frac{(-1)^{|\be+\delta+\nu|}}{(2i)^{|\al+\be+\delta+\nu|}\al!\be!\delta!\nu!}S^{(\al)}_{(\be+\delta)}A^{(\be)}_{(\al+\nu)}\lr{\xi}^{(\delta+\nu)}\\
&-\sum_{1\leq |\al+\be+\delta+\nu|\leq 2}\frac{(-1)^{|\al|}}{(2i)^{|\al+\be+\delta+\nu|}\al!\be!\delta!\nu!}(A^*)^{(\be)}_{(\al+\nu)}S^{(\al)}_{(\be+\delta)}\lr{\xi}^{(\delta+\nu)}+R\\
&=\sum_{|\al+\be+\delta+\nu|=1}+\sum_{|\al+\be+\delta+\nu|=2}+R
=K_1+K_2+R,
\end{aligned}
\]
where $R\in S^{-2}$ and $K_j\in S^{1-j}$ denotes the sum over $|\al+\be+\delta+\nu|=j$. Note that 
\[
S^{(\al)}_{(\be+\delta)}=\left(\begin{array}{ccc}
0&0&-a^{(\al)}_{(\be+\delta)}\\
0&2a^{(\al)}_{(\be+\delta)}&3b^{(\al)}_{(\be+\delta)}\\
-a^{(\al)}_{(\be+\delta)}&3b^{(\al)}_{(\be+\delta)}&(a^2)^{(\al)}_{(\be+\delta)}\end{array}\right)
\]
and if $|\al+\be+\nu|\neq 0$
\[
A^{(\be)}_{(\al+\nu)}=\left(\begin{array}{ccc}
0&a^{(\be)}_{(\al+\nu)}&b^{(\be)}_{(\al+\nu)}\\
0&0&0\\
0&0&0\end{array}\right)
\]
so that
\begin{equation}
\label{eq:itoa}
S^{(\al)}_{(\be+\delta)}A^{(\be)}_{(\al+\nu)}=\left(\begin{array}{ccc}
0&0&0\\
0&0&0\\
0&-a^{(\al)}_{(\be+\delta)}a^{(\be)}_{(\al+\nu)}&-a^{(\al)}_{(\be+\delta)}b^{(\be)}_{(\al+\nu)}\end{array}\right).
\end{equation}
If $|\al+\be+\nu|=0$ then
\begin{equation}
\label{eq:itob}
S_{(\delta)}A=\left(\begin{array}{ccc}
0&-a_{(\delta)}&0\\
2a_{(\delta)}&3b_{(\delta)}&0\\
3b_{(\delta)}&(a^2)_{(\delta)}-a a_{(\delta)}&-a_{(\delta)}b\end{array}\right).
\end{equation}
From Lemma \ref{lem:sou} it follows that for $|\al+\be+\delta+\nu|=1$
\begin{align*}
S^{(\al)}_{(\be+\delta)}A^{(\be)}_{(\al+\nu)}\lr{\xi}^{(\delta+\nu)}-(A^*)^{(\be)}_{(\al+\nu)}S^{(\al)}_{(\be+\delta)}\lr{\xi}^{(\delta+\nu)}=\left(\begin{array}{ccc}
0&0&0\\
0&0&\co(a)\\
0&\co(a)&\co(a^{3/2})\end{array}\right)
\end{align*}
and
\[
S_{(\delta)}A\lr{\xi}^{(\delta)}-A^*S_{(\delta)}\lr{\xi}^{(\delta)}=\left(\begin{array}{ccc}
0&\co(\sqrt{a})&\co(a)\\
\co(\sqrt{a})&\co(a)&\co(a^{3/2})\\
\co(a)&\co(a^{3/2})&\co(a^2)\end{array}\right).
\]
Therefore we deduce
\begin{equation}
\label{eq:sakua}
S+\epsilon_1 t{\mathsf{Im}}K_1=\left(\begin{array}{ccc}
3&\epsilon_1 \co(t\sqrt{a})&-a+\epsilon_1 \co(ta)\\
\epsilon_1 \co(t\sqrt{a})&2a+\epsilon_1 \co(ta)&3b+\epsilon_1 \co(ta)\\
-a+\epsilon_1 \co(ta)&3b+\epsilon_1 \co(ta)&a^2+\epsilon_1 \co(ta^{3/2})\end{array}\right).
\end{equation}
On the other hand,
\begin{equation}
\label{eq:sakub}
t\dif_t S=\left(\begin{array}{ccc}
0&0&\co(t)\\
0&\co(t)&\co(ta^{1/2})\\
\co(t)&\co(ta^{1/2})&\co(ta)\end{array}\right).
\end{equation}
This proves that there exists $N_1>0$ such that one has 
\begin{equation}
\label{eq:iru}
N_1S+ t{\mathsf{Im}}K_1- t \dif_t S=N_1(S+\epsilon_1t{\mathsf{Im}}K_1- \epsilon_1t \dif_t S)\gg 0,
\end{equation}
where $\epsilon_1=N_1^{-1}$.
\begin{lem}
\label{lem:Cinfty}
For any $\epsilon>0$ there exists $C_{\epsilon}>0$ such that
\[
\begin{aligned}
&{\mathsf{Re}}\big((N_1S- t\dif_tS-i tK_1)U,U\big)\\
&\geq -\epsilon t\big(\sum_{j=1}^2\|U_j\|^2+(aU_3,U_3)\big)-C_{\epsilon}t^{-1}\|\lr{D}^{-1}U\|^2.
\end{aligned}
\]
\end{lem}
\begin{proof} 
Setting $Q=N_1S- t\dif_tS+ t{\mathsf{Im}}K_1\in S^0$ one has $Q\gg 0$ by \eqref{eq:iru}. Taking \eqref{eq:sGar} into account and repeating the same arguments proving \eqref{eq:FeP}, it is enough to show
\[
\psi_{\al,\be}Q^{(\al)}_{(\be)}=
\left(\begin{array}{ccc}
\co(\lr{\xi}^{-1})&\co(\lr{\xi}^{-1})&\co(\lr{\xi}^{-1})\\
\co(\lr{\xi}^{-1})&\co(\lr{\xi}^{-1})&\co(\lr{\xi}^{-1})\\
\co(\lr{\xi}^{-1})&\co(\lr{\xi}^{-1})&\co(\lr{\xi}^{-1}\sqrt{a})\end{array}\right)
\]
for $2\leq |\al+\be|\leq 3$. To check this it suffices to apply \eqref{eq:sakua}, \eqref{eq:sakub} and $t=\co(a)$.
\end{proof}
We turn to $K_2$. From \eqref{eq:itoa} and \eqref{eq:itob} if follows that 
\[
K_2=\left(\begin{array}{ccc}
\co(\lr{\xi}^{-1})&\co(\lr{\xi}^{-1})&\co(\lr{\xi}^{-1})\\
\co(\lr{\xi}^{-1})&\co(\lr{\xi}^{-1})&\co(\lr{\xi}^{-1})\\
\co(\lr{\xi}^{-1})&\co(\lr{\xi}^{-1})&\co(\lr{\xi}^{-1}\sqrt{a})\end{array}\right).
\]
Repeating the same arguments as above, we get
\[
|(K_2U,U)|\leq \epsilon \sum_{j=1}^2\|U_j\|^2+C\epsilon^{-1}\|\lr{D}^{-1}U\|^2+\epsilon \|TU_3\|^2,
\]
where $T\in S^0$ and $|T|\leq C\sqrt{a}$. Applying the Fefferman-Phong inequality again, one concludes
\[
|(K_2U,U)|\leq \epsilon \big(\sum_{j=1}^2\|U_j\|^2+(aU_3,U_3)\big)+C\epsilon^{-1}\|\lr{D}^{-1}U\|^2.
\]
Noting that $(3,3)$-entry of $A$ is zero, it is clear that 
\begin{align*}
t^{-N -1}\lambda\big|\big((\lr{D}^{-2}A\lr{D}-\lr{D}A^*\lr{D}^{-2})U,U\big)\big|
\leq \epsilon t^{-N} \sum_{j=1}^2\|U_j\|^2+C\lambda^2\epsilon^{-1}t^{-N -2}\|\lr{D}^{-1}U\|^2.
\end{align*}
On the other hand, it is easy to see that
\[ 
t^{-N -1}\lambda\big|\big((\lr{D}^{-2}B-B^*\lr{D}^{-2})U,U\big)\big| \leq C_1 \lambda t^{-N - 1}\|\lr{D}^{-1} U\|^2
\]
with a constant $C_1 > 0$ independent of $\lambda$ and $t$.
Therefore from the above estimates one deduces
\begin{equation}
\label{eq:shima}
\begin{split}
\dif_t{\mathsf{Re}}(t^{-N}e^{-\gam t}{\tilde S}U,U)\leq -2{\mathsf{Im}}(t^{-N}e^{-\gam t}{\tilde S}F,U)
-(N-N_1)t^{-N-1}e^{-\gam t}{\mathsf{Re}}\big({\tilde S}U,U)\\
+\Bigl[C_{\ep}-\lambda\Bigl( (N + 1) - \lambda C \ep^{-1}\Bigr)\Bigr]t^{-N-2}e^{-\gam t}\|\lr{D}^{-1}U\|^2
\\
+2\epsilon t^{-N}e^{-\gam t}\big(\sum_{j=1}^2\|U_j\|^2+(aU_3,U_3)\big)-2t^{-N}e^{-\gam t}{\mathsf{Im}}\,({\tilde S}BU,U)\\
-(\gam - \lambda C_1) t^{-N-1}e^{-\gam t}\|\lr{D}^{-1}U\|^2,
\end{split}
\end{equation}
where $C$ is independent of $t$.

Consider $2{\mathsf{Im}}({\tilde S}BU,U)$ and recall that ${\tilde S}\gg 0$ by Corollary \ref{cor:seiti}. Consequently,
\begin{equation}
\label{eq:raiu}
\begin{split}
&2|({\tilde S}BU,U)|\leq N^{-1/2}(t{\tilde S}BU,BU)+N^{1/2}(t^{-1}{\tilde S}U,U)\\
&=N^{-1/2}(t^{-1}t^2B^*{\tilde S}BU,U)+N^{1/2}(t^{-1}{\tilde S}U,U)\\
&\leq N^{-1/2}(t^{-1}t^2B^*SBU,U)+N^{1/2}(t^{-1}\tilde{S}U,U)+ C\lambda N^{-1/2}\|\lr{D}^{-1}U\|^2.
\end{split}
\end{equation}
Note that
\[
B^*\#S\#B-3 \left(\begin{array}{c}
{\bar b}_{10}\\
{\bar b}_{11}\\
{\bar b}_{12}\end{array}\right)(b_{10},b_{11},b_{12})\in S^{-1}.
\]
\begin{lem}
\label{lem:parao}There exists $N_2>0$ such that for any $\epsilon>0$ there exists $D_{\epsilon}>0$ such that
\[
{\mathsf{Re}}\big((N_2S-t^2B^*SB)U,U\big)\geq -\epsilon t(\sum_{j=1}^2\|U_j\|^2+(aU_3,U_3))-D_{\epsilon}t^{-1}\|\lr{D}^{-1}U\|^2.
\]
\end{lem}
\begin{proof}Recall
\[
S-\epsilon_1 t^2B^*SB=\left(\begin{array}{ccc}
3+\epsilon_1 \co(t^2)&\epsilon_1 \co(t^2)&-a+\epsilon_1 \co(t^2)\\
\epsilon_1 \co(t^2)&2a+\epsilon_1 \co(t^2)&3b+\epsilon_1 \co(t^2)\\
-a+\epsilon_1 \co(t^2)&3b+\epsilon_1 \co(t^2)&a^2+\epsilon_1 \co(t^2)\end{array}\right)
\]
which proves $S-\epsilon_1 t^2B^*SB\gg 0$ with some $\epsilon_1>0$. Then the rest of the proof is just a repetition of the proof of Lemma \ref{lem:Cinfty}.
\end{proof}
According to Lemma \ref{lem:parao} and \eqref{eq:raiu}, one has
\begin{equation}
\label{eq:ike}
\begin{split}
2|({\tilde S}BU,U)|\leq N_2^{1/2}(t^{-1}\tilde{S}U,U)+\epsilon t(\sum_{j=1}^2\|U_j\|^2+(aU_3,U_3))\\
+D_{\epsilon}t^{-1}\|\lr{D}^{-1}U\|^2+ C\lambda N_2^{-1/2}\|\lr{D}^{-1}U\|^2.
\end{split}
\end{equation}
From \eqref{eq:shima} and \eqref{eq:ike} it follows that
\[
\begin{aligned}
\dif_t{\mathsf{Re}}(t^{-N}e^{-\gam t}{\tilde S}U,U)\leq -2{\mathsf{Im}}(t^{-N}e^{-\gam t}{\tilde S}F,U)\\
-(N-N_1-N_2^{1/2})t^{-N-1}e^{-\gam t}{\mathsf{Re}}({\tilde S}U,U)\\
+\Bigl[C_{\ep}-\lambda\Bigl( (N + 1) - \lambda C \ep^{-1} \Bigr)\Bigr]t^{-N-2}e^{-\gam t}\|\lr{D}^{-1}U\|^2\\
+3\epsilon t^{-N}e^{-\gam t}\big(\sum_{j=1}^2\|U_j\|^2+(aU_3,U_3)\big)\\
-(\gam- D_{\ep}-C_1 \lambda - C \lambda t N_2^{-1/2})t^{-N-1}e^{-\gam t}\|\lr{D}^{-1}U\|^2.
\end{aligned}
\]
Note that
\[
\begin{aligned}
2|(t^{-N}e^{-\gam t}{\tilde S}F,U)|\leq 2(t^{-N+1}e^{-\gam t}{\tilde S}F,F)^{1/2}(t^{-N-1}e^{-\gam t}{\tilde S}U,U)^{1/2}\\
\leq (t^{-N+1}e^{-\gam t}{\tilde S}F,F)+(t^{-N-1}e^{-\gam t}{\tilde S}U,U).
\end{aligned}
\]
Denote $N^*=N_1+ 2N_2^{1/2} + 2$ and we choose $0 <3\epsilon\leq \delta$. The term with coefficient $C_{\ep}$ can be absorbed by 
the term $- {\mathsf{Re}}\,(\tilde{S}U, U)$ by applying Corollary 3.2 and taking $\lambda /2 > C_{\ep}$. Fixing $\ep$ and $\lambda$, we choose $N > N^*$ so that
$$N  +1 > \lambda C \ep^{-1}.$$

Finally we choose $\gamma$ such that $\gamma-  D_{\ep} -C_1 \lambda- C \lambda N_2^{-1/2}T\geq 0$. Then we have
\begin{equation}
\label{eq:keiyaku}
\dif_t{\mathsf{Re}}(t^{-N}e^{-\gam t}{\tilde S}U,U)\leq (t^{-N+1}e^{-\gam t}{\tilde S}F,F)-(N-N^*){\mathsf{Re}}\,(t^{-N-1}e^{-\gam t}{\tilde S}U,U).
\end{equation}
Integrating \eqref{eq:keiyaku} in $\tau$ from $\epsilon>0$ to $t$ and taking  Corollary \ref{cor:seiti:2} into account, one obtains
\begin{prop}
\label{pro:senken} Assume the condition $(3.10)$. Let $U \in C^{\infty}(\R_t: C_0^{\infty}(\R^n))).$ Then there exist $\delta>0$ and $C>0$ such that
\[
\begin{aligned}
\delta t^{-N+2}e^{-\gam t}\|U(t)\|^2+\delta(N-N^*)\int_{\epsilon}^t\tau^{-N+1}e^{-\gam \tau}\|U(\tau)\|^2 d\tau\\
\leq C\epsilon^{-N-1}e^{-\gam \epsilon}\|U(\epsilon)\|^2+\int_{\epsilon}^t(\tau^{-N+1}e^{-\gam \tau}{\tilde S}F(\tau),F(\tau))d\tau.
\end{aligned}
\]
\end{prop}
\begin{rem} In the case when the condition $(H)$ is globally satisfied by an argument similar to that used for Lemma $3.4$ we may establish the inequality

\[
{\mathsf{Re}}(SU,U)\geq \delta t^2\big(\sum_{j=1}^2\|U_j\|^2+(aU_3,U_3)\big)-Ct^{-2}\|\lr{D}^{-1}U\|^2.
\]
However with this weaker version of the Fefferman-Phong inequality we must change the weight $t^{-N}$ by $e^{ \frac{N}{t}}$ and  use the energy $(e^{\frac{N}{t} }e^{- \gamma t} (S + \lambda t^{-2} \la D \ra^{-2})U, U)$. This will be a subject of a further work.

\end{rem}
\section{Microlocal energy estimates}
%%%%%%%%%%%%
%%Partition of unity
%%%%%%%%%%%%%
We turn to the original differential operator 
\[
P=D_t^3-\sum_{|\al|=2}a_{\al}(t,x)D_x^{\al}D_t+\sum_{|\al|=3}a_{\al}(t,x)D_x^{\al}+\sum_{|\be|+j\leq 2}b_{j\be}(t,x)D_x^{\be}D_t^j
\]
with which we are working. Let $\{\chi_{\al}\}$ be a finite partition of unity with $\chi_{\al}(x,\xi)\in S^0$ so that
\[
\sum_{\alpha} \chi^2_{\al}(x,\xi)=\chi^2(x),
\]
where $\chi(x)=1$ for $|x|\leq r_1$ and $0$ for $|x|\geq r_2$. Let $u \in C^{\infty}(\R_t: C_0^{\infty}(\R^{n}))$ be such that $u = 0$ for $|x| \geq r_1.$ We can assume that  \eqref{eq:2.10} is verified globally. Indeed we have the following
\begin{lem}
\label{lem:kakutyo}
Assume that \eqref{eq:2.10} is satisfied in a conic neighborhood of $(0,\xi_0)$. Then there exist extensions of $\alpha(x,\xi)$, ${\tilde a}(t,x,\xi)$ and $b(t,x,\xi)$ to $S^0$ such that \eqref{eq:2.10} holds globally.
\end{lem}
\begin{proof}
Assume that \eqref{eq:2.10} is satisfied in a conic neighborhood $W$ of $(0,\xi_0)$. Choose conic neighborhoods $U$, $V$ of $(0,\xi_0)$ such that $U\Subset V\Subset W$. Take $0\leq\chi_i(x,\xi)\in S^0$, $i=1,2$ such that $\chi_1=1$ on $V$ and $0$ outside $W$ and $\chi_2=0$ on $U$ and $1$ outside $V$. Extend ${\tilde a}$ outside $W$ so that ${\tilde a}\in S^0$ and ${\tilde a}\geq c_0>0$. Denote the extended symbol again by ${\tilde a}$. Define ${\tilde \alpha}=\chi_1\alpha+M\chi_2\in S^0$ and ${\tilde b}=\chi_1b\in S^0$,  where $M>0$ is a positive constant. If $(x,\xi)\in V$ we see that
\begin{align*}
\Delta = 4(t+\alpha+M\chi_2)^3{\tilde a}^3-27b^2\geq 4(t+\alpha)^3{\tilde a}^3-27b^2\\
+12\big((t+\alpha)^2M\chi_2+(t+\alpha)M^2\chi_2^2\big){\tilde a}^3\\
\geq \delta t(t+\alpha)^2+c_0't(t+\alpha)M\chi_2+c_0'tM^2\chi_2^2\\
\geq\delta' t(t+\alpha+M\chi_2)^2=\delta' t(t+{\tilde \alpha})^2.
\end{align*}
If $(x,\xi)$ is outside $W$,  it is clear that $\Delta =4(t+M)^3{\tilde a}^3\geq \delta't(t+M)^2=\delta't(t+{\tilde \alpha})^2$. On the other hand, if $(x,\xi)\in W\setminus V$ one has
\begin{align*}
\Delta =4(t+\chi_1\alpha+M)^3{\tilde a}^3-27(\chi_1b)^2\\
\geq 
4 t\big((t+\chi_1\alpha)^2+3(t+\chi_1\alpha)M+3tM^2\big){\tilde a}^3+4M^3{\tilde a}^3-27(\chi_1b)^2\\
\geq \delta' t(t+\chi_1\alpha+M)^2=\delta' t(t+{\tilde\alpha})^2
\end{align*}
choosing $M$ such that $4M^3{\tilde a}^3\geq 27(\chi_1b)^2$.
\end{proof}
We consider pseudo-differential operators with symbols $\chi_{\alpha}(x, \xi)$ and without loss of generality we can assume that $\chi_{\alpha}(x, \xi) = 0$ for $|x| \geq r_2$. First we shall  estimate $\chi_{\al}u \in C_0^{\infty}(\R^{n+1}).$ To do this we observe that
\[
P\chi_{\al}u-\chi_{\al}Pu=R_{\al}u,
\]
where $R_{\al}$ has the form
\begin{equation}
\label{eq:kaji}
c_{1, \al}D_t^2u+c_{2, \al}D_t\lr{D}u+c_{3,\al}\lr{D}^2u=\sum_{j=1}^3c_{j, \al}U_j.
\end{equation}
Here $U={^t}(D_t^2u,D_t\lr{D}u,\lr{D}^2u)=(U_1,U_2,U_3)$ again and
 $c_{j, \al}\in S^0$ are symbols vanishing for $|x| \geq r_2.$  Thanks to Lemma \ref{lem:kakutyo} extending $P$ outside the support of $\chi_{\al}$, we can assume that $P$ has the form \eqref{eq:takata} and \eqref{eq:jyokenE} is satisfied. Then  with $U_{\al}={^t}(D_t^2\chi_{\al} u,D_t\lr{D}\chi_{\al}u,\lr{D}^2\chi_{\al}u)$ the equation can be written as
\begin{equation}
\label{eq:shino}
D_tU_{\al}=(A\lr{D}+B)U_{\al}+F_{\al},\quad F_{\al}={^t}\Bigl(\chi_{\al}f+\sum_{j=1}^3c_{j, \al}U_j,0,0\Bigr),
\end{equation}
 From Proposition \ref{pro:senken}  we deduce
\begin{equation}
\label{eq:onisi}
\begin{split}
\delta t^{-N+2}e^{-\gam t}\|U_{\al}(t)\|^2+\delta (N-N^*)\int_{\epsilon}^t\tau^{-N+1}e^{-\gam \tau}\|U_{\al}(\tau)\|^2d\tau\\
\leq C\epsilon^{-N-1}e^{-\gam \epsilon}\|U_{\al}(\epsilon)\|^2+\int_{\epsilon}^t(\tau^{-N+1}e^{-\gam \tau}{\tilde S}F_{\al}(\tau),F_{\al}(\tau))d\tau.
\end{split}
\end{equation}
From \eqref{eq:shino} it follows that $
|({\tilde S}F_{\al},F_{\al})|\leq C(\|f\|^2+\|U\|^2)$ and hence 
\begin{equation}
\label{eq:ota}
\sum_{\al}\int_{\epsilon}^t(\tau^{-N+1}e^{-\gam \tau}{\tilde S}F_{\al}(\tau),F_{\al}(\tau))d\tau\leq C \int_{\epsilon}^t\tau^{-N+1}e^{-\gam \tau}(\|f(\tau)\|^2+\|U(\tau)\|^2)d\tau.
\end{equation}
Since
\[
\sum_{\al}\chi_{\al}\#\lr{\xi}^k\#\chi_{\al}=\chi\#\lr{\xi}^k\#\chi+R_k,\quad R_k\in S^{k-2},
\]
there is $C>0$ such that
\[
\begin{aligned}
\sum_{\al}\|U_{\al}\|^2=\sum_{\al}\big\{(\chi_{\al}D_t^2u,\chi_{\al}D_t^2u)+(\lr{D}\chi_{\al}D_tu,\lr{D}\chi_{\al}D_tu)\\
+(\lr{D}^2\chi_{\al}u,\lr{D}^2\chi_{\al}u)\big\}\geq \|U\|^2-C\|\lr{D}^{-1}U\|^2
\end{aligned}
\]
because $\chi D_t^ku=D_t^ku.$ Therefore summing up \eqref{eq:onisi} over $\al$ and choosing $N$ such that  the second term on the left-hand side of \eqref{eq:onisi} absorbs \eqref{eq:ota}, one can find $c> 0$ and $C_1,C_2>0$ such that 
\begin{equation}
\label{eq:karei}
\begin{split}
c\,t^{-N+2}e^{-\gam t}\|U(t)\|^2+c\int_{\epsilon}^t\tau^{-N+1}e^{-\gam \tau}\|U(\tau)\|^2d\tau\\
\leq C_1\Big(\epsilon^{-N-1}e^{-\gam \epsilon}\|U(\epsilon)\|^2+\int_{\epsilon}^t\tau^{-N+1}e^{-\gam \tau}\|f(\tau)\|^2d\tau\Big)\\
+C_2\Big(t^{-N+2}e^{-\gam t}\|\lr{D}^{-1}U(t)\|^2+\int_{\epsilon}^t\tau^{-N+1}e^{-\gam \tau}\|\lr{D}^{-1}U(\tau)\|^2d\tau\Big).
\end{split}
\end{equation}
Let $\chi(\xi)\in C_0^{\infty}(\R^n: [0, 1])$ be a function which is equal to $1$ in $|\xi|\leq 1$ and to $0$ for $|\xi|\geq 2$. Set 
\[
\chi_{\nu}(\xi)=\chi(\nu\xi),
\]
where $\nu>0$ is a small parameter. Then we have the following
\begin{lem}
\label{lem:kusakari}Assume that $U$ satisfies \eqref{eq:redE}. Then there exists a constant $C>0$ such that the function  $U_{\nu}=\chi_{\nu}U$ satisfies the estimate
\begin{equation}
\label{eq:kusa}
\begin{split}
t^{-N+1}e^{-\gam t}\|U_{\nu}(t)\|^2+(\gam-C\nu^{-1})\int_{\epsilon}^t\tau^{-N+1}e^{-\gam\tau}\|U_{\nu}(\tau)\|^2d\tau\\
\leq \epsilon^{-N+1}e^{-\gam \epsilon}\|U(\epsilon)\|^2
+C\nu^2\int_{\epsilon}^t\tau^{-N+1}e^{-\gam\tau}\|U(\tau)\|^2d\tau\\
+C\int_{\epsilon}^t\tau^{-N+1}e^{-\gam\tau}\|f(\tau)\|^2d\tau.
\end{split}
\end{equation}
\end{lem}
\begin{proof}Note that one has 
\begin{equation}
\label{eq:petk}
D_tU_{\nu}=A_{\nu}U_{\nu}+\chi_{\nu}F+R_{\nu}U
\end{equation}
with $A_{\nu}=(A\lr{D}+B)\chi_{\nu/2}$ and $R_{\nu}=[\chi_{\nu},A\lr{D}+B]$, where we have used the equality $\chi_{\nu/2}\chi_{\nu}=\chi_{\nu}$. It is clear that $A_{\nu}, R_{\nu}\in S^0$ and moreover $\nu A_{\nu}, \nu^{-1}R_{\nu}\in S^0$ uniformly in $\nu>0$ because $|\dif_{\xi}^{\al}\chi_{\nu}(\xi)|\leq C_{\al}(\nu^{-1}+|\xi|)^{-|\al|}$ and $\lr{\xi}\leq C\nu^{-1}$ on the support of $\chi_{\nu/2}$. This shows that there is $C>0$ independent of $\nu$ such that
\begin{equation}
\label{eq:ikemura}
\|A_{\nu}U\|\leq C\nu^{-1}\|U\|,\qquad \|R_{\nu}U\|\leq C\nu\|U\|.
\end{equation}
From \eqref{eq:petk} we have
\[
\begin{aligned}
\dif_t\big(t^{-N+1}e^{-\gam t}\|U_{\nu}\|^2\big)+(N-1)t^{-N}e^{-\gam t}\|U_{\nu}\|^2
+\gam t^{-N+1}e^{-\gam t}\|U_{\nu}\|^2\\=-2t^{-N+1}e^{-\gam t}{\mathsf{Im}}\,(A_{\nu}U_{\nu}+\chi_{\nu}F+R_{\nu}U,U_{\nu})\\
\leq Ct^{-N+1}e^{-\gam t}\big(\nu^{-1}\|U_{\nu}\|^2+\|f\|^2+\nu^2\|U\|^2\big)
\end{aligned}
\]
by using \eqref{eq:ikemura}. Integrating this inequality, we obtain the assertion.
\end{proof}
Since $\lr{\xi}^{-1}(1-\chi_{\nu})\leq \nu$ writing $\lr{D}^{-1}U=\lr{D}^{-1}\chi_{\nu}U+\lr{D}^{-1}(1-\chi_{\nu})U$, it is clear that
\[
\|\lr{D}^{-1}U\|\leq \|U_{\nu}\|+\nu\|U\|.
\]
Thus the second term on the right-hand side of \eqref{eq:karei} is bounded by
\[
C_2t^{-N+2}e^{-\gam t}\big(\nu^2\|U(t)\|^2+\|U_{\nu}(t)\|^2\big)+C_2\int_{\epsilon}^t\tau^{-N+1}e^{-\gam\tau}\big(\nu^2\|U(\tau)\|^2+\|U_{\nu}(\tau)\|^2\big)d\tau.
\]
Thus choosing $\nu>0$ small and $\gamma$ such that $\gam\geq C\nu^{-1}+1$, from   \eqref{eq:kusa} and \eqref{eq:karei} we obtain the following
\begin{thm}
\label{pro:energy}Let $U \in C^{\infty}(\R_t: C_0^{\infty}(\R^{n})).$ Assume that for every point $(x_0,\xi_0) \in T^*(U) \setminus \{0\}$ with $\alpha(x_0, \xi_0) = 0$ there exist a conic neighborhood $W \subset T^*(U) \setminus \{0\}$ and $T(x_0, \xi_0) > 0$ such that the condition (E) is satisfied for $0 \leq t \leq T(x_0, \xi_0)$ and $(x, \xi) \in W$. Then there exist $c>0,\: T_0 > 0,\: C>0$ and $N \in \N$ such that for $0 < \epsilon < t \leq T_0$ we have
\begin{equation} 
\label{eq:4.9}
\begin{split}
c\,t^{-N+2}e^{-\gam t}\|U(t)\|^2+c\int_{\epsilon}^t\tau^{-N+1}e^{-\gam \tau}\|U(\tau)\|^2d\tau \\
\leq C\epsilon^{-N-1}e^{-\gam \epsilon}\|U(\epsilon)\|^2+C\int_{\epsilon}^t\tau^{-N+1}e^{-\gam \tau}\|f(\tau)\|^2d\tau.
\end{split}
\end{equation}
\end{thm}
We may absorb the weight $\tau^{-N}$ passing to energy estimates with a loss of derivatives. Denote ${\mathcal P}=D_t-A\lr{D}-B$ and consider ${\mathcal P}U=F$ for $U \in C^{\infty}(\R_t: C_0^{\infty}(\R^n))$.
Differentiating ${\mathcal P}U = F$ with respect to $t$, we can find the functions
$D _t^j U(0, x) = U_j(x) \in C_0^{\infty}(\R^n)$. Next we set
$$
U_M(t,x) = \sum_{j=0}^M \frac{1}{j!} U_j(x)(i t)^j.
$$
Then $W = U - U_M \in C^{\infty}(\R_t: C_0^{\infty}(\R^n))$ satisfies ${\mathcal P}W = F_M$, where
$$
D_t^j F_M(0, x) = 0,\: j = 0,1,\ldots, M-1,\quad D_t^j W(0, x) = 0,\:j = 0,1,\ldots,M.
$$
Note that
\[
\|U_M(t,\cdot)\|\leq C_M\big(\|U(0)\|_{H^M(\R^n)}+\sum_{j=0}^{M-1}\|D_t^jF(t,\cdot)\|_{H^{M-1-j}(\R^n)}\big)
\]
and for $2M\geq N$
\[
 \int_{\ep}^te^{-\gamma \tau}\tau^{-N +1} \|F_M(\tau, \cdot)\|^2 d\tau \leq C_{N,M} \int_0^te^{-\gamma \tau}\|D_t^MF(\tau, \cdot)\|^2 d\tau. 
\]
Choosing $2M \geq  N+1$,  we apply Theorem \ref{pro:energy} for $W$. Since $\epsilon^{-N-1}\|W(\epsilon)\|^2\to 0$ as $\epsilon\to 0$,  we obtain 
\begin{cor}
\label{cor:hyoka}
 Under the assumptions of Theorem {\rm 4.1} there exist $N\in\N$,   $T>0$, $C>0$ such that for $0 < t \leq T$ we have the estimate
\begin{equation} \label{eq:4.9}
\int_0^t e^{-\gam \tau}\|U(\tau, \cdot)\|^2d\tau
\leq C\Big(\|U(0,\cdot)\|_{H^N(\R^n)}+\sum_{j=0}^N\int_0^t e^{-\gam \tau}\|D_t^jf(\tau,\cdot)\|_{H^{N-j}(\R^n)}^2d\tau\Big)
\end{equation}
for $U \in C^{\infty}(\R_t:\:C_0^{\infty}(\R^n))$ and  $\gam \geq \gam_0$.
\end{cor}

By the same argument we can obtain an estimate for the adjoint operator $P^*$ which has the same principal symbol as $P$. These estimates imply by a standard method the well-posedness of the Cauchy problem for $P$ and we get the following

\begin{cor} Under the assumptions of Theorem $4.1$  the operator $P$ is strongly hyperbolic.

\end{cor}

We can also obtain the result by another way. Consider the operator $P_{\epsilon}$  obtained from $P$ replacing $\alpha$ by $\alpha+\epsilon$ with $\epsilon>0$. Then it is clear that $P_{\epsilon}$ is strictly hyperbolic in $t\geq 0$. On the other hand, as it was  seen in the proof of Lemma \ref{lem:kakutyo}, the condition \eqref{eq:2.10} holds with some $\delta>0$ independent of $\epsilon>0$. It is also easy to check that Lemma \ref{lem:sou} holds uniformly in $\epsilon>0$. Therefore thanks to Corollary \ref{cor:hyoka} we obtain a priori estimates which is uniform in  $\epsilon>0$ for the solution $u^{\epsilon}$ to the Cauchy problem  
\[
P_{\epsilon}u^{\epsilon}=f,\quad \partial_t^ju^{\epsilon}(0,x)=u_j(x),\quad j=0,1,2.
\]
Thus there is a subsequence of $\{u^{\epsilon}\}$ convergent to the solution to the Cauchy problem for $P$.

%%%%%%%%%%
%%%%Appendix
%%%%%%%%%%%

\renewcommand{\thelem}{A.\arabic{lem}}
 \renewcommand{\theequation}{A.\arabic{equation}}
  % redefine the command that creates the equation no.
  \setcounter{equation}{0}  % reset counter 
  \section*{Appendix}  % use *-form to suppress numbering
  
In this appendix we give a proof of \eqref{eq:sGar}. Assume that  $p(x,\xi)$ is a $m\times m$ matrix with entries in $ S^0_{1,0}(\R^n\times \R^n)$  such that $p\gg 0$. To define $p_F$ we follow \cite{Fr}. Choose $0\leq q(\sigma)\in C_0^{\infty}(\R^n)$ such that
\[
q(-\sigma)=q(\sigma),\quad \int q^2(\sigma)d\sigma=1.
\]
Let us put
\[
F(\xi,\zeta)=q((\zeta-\xi)\lr{\xi}^{-1/2})\lr{\xi}^{-n/4}.
\]
We first check that one can write
\begin{equation}
\label{eq:a_1}
\begin{split}
\dif_{\xi}^{\be}F(\xi,\zeta)=\lr{\xi}^{-n/4}\sum_{|\gamma|\leq |\be|,\gamma_1\leq \gamma}((\zeta-\xi)\lr{\xi}^{-1/2})^{\gamma_1}\\
\times (\dif_{\sigma}^{\gamma}q)((\zeta-\xi)\lr{\xi}^{-1/2})\psi_{\be,\gamma,\gamma_1},
\end{split}
\end{equation}
where $\psi_{\be,\gamma,\gamma_1}\in S^{-(|\be|-|\gamma-\gamma_1|/2)}_{1,0}$. Indeed noting that 
\[
\dif_{\xi}^{\al}\lr{\xi}^s=\lr{\xi}^s\psi,\quad \psi\in S^{-|\al|}_{1,0},
\]
the assertion follows from the condition on $|\be|$. Let us put
\[
p_F(\eta,y,\xi)=\int F(\eta,\zeta)p(y,\zeta)F(\xi,\zeta)d\zeta
\]
which is called the Friedrichs part of $p$ and define the operator $p_F$ by
\[
(p_Fu)(x)=(2\pi)^{-2n}\int e^{i(x-x')\xi+i x'\xi'}p_F(\xi,x',\xi'){\hat u}(\xi')d\xi'dx'd\xi.
\]
We first remark that
\begin{eqnarray*}
(p_Fu,v)=\int p(x',\zeta)\Big(\int e^{ix'\xi'}F(\xi',\zeta){\hat u}(\xi')d\xi'\Big)\overline{\Big(\int e^{ix'\xi}F(\xi,\zeta){\hat v}(\xi)d\xi\Big)}dx'd\zeta
\end{eqnarray*}
which implies
\begin{equation}
\label{eq:a_4}
(p_Fu,u)=\int p(x',\zeta)\Big|\int e^{ix'\xi'}F(\xi',\zeta){\hat u}(\xi')d\xi'\Big|^2dx'd\zeta\geq 0
\end{equation}
since $p(x',\zeta)\gg 0$.
We next check that
\begin{equation}
\label{eq:a_2}
p_Fu=q^wu,\;\;q(x,\xi)=(2\pi)^{-n}\int e^{-iy\eta}p_F\Bigl(\xi +\frac{\eta}{2},x+y,\xi-\frac{\eta}{2}\Bigr)dyd\eta.
\end{equation}
In fact we see that
\begin{eqnarray*}
p_Fu=(2\pi)^{-n}\int e^{ix\xi'}{\hat u}(\xi')d\xi'\Big((2\pi)^{-n}\int e^{i(x-x')\xi-ix\xi'+ix'\xi'}p_F(\xi,x',\xi')dx'd\xi\Big)\\=(2\pi)^{-n}\int e^{ix\xi}\Big((2\pi)^{-n}\int e^{-ix'\xi'}p_F(\xi+\xi',x+x',\xi)d\xi'dx'\Big){\hat u}(\xi)d\xi={\tilde p}(x,D)u,
\end{eqnarray*}
where
\begin{equation}
\label{eq:a_3}
{\tilde p}(x,\xi)=(2\pi)^{-n}\int e^{-iy\eta}p_F(\xi+\eta,x+y,\xi)d\eta dy.
\end{equation}
Since we have ${\tilde p}(x,D)=q^w(x,D)$ with
\[
q(x,\xi)=(2\pi)^{-n}\int e^{iz\zeta}{\tilde p}\Bigl(x+\frac{z}{\sqrt{2}},\xi+\frac{\zeta}{\sqrt{2}}\Bigr)dzd\zeta,
\]
inserting (\ref{eq:a_3}) into the above formula, we get the desired assertion.

\begin{lem}
\label{lem:a_1}
We have
\[
q(x,\xi)=p(x,\xi)+\sum_{2\leq |\al+\be|\leq 3}\Psi_{\al,\be}(\xi)p^{(\al)}_{(\be)}(x,\xi)+r(x,\xi),\quad r\in S^{-2}_{1/2,0},
\]
where $\Psi_{\al,\be}\in S^{(|\al|-|\be|)/2}_{1,0}$.
\end{lem}
From (\ref{eq:a_2}) one can write
\begin{eqnarray*}
q(x,\xi)=(2\pi)^{-n}\int e^{-iy\eta}(iy)^{\be}\Big\{\sum_{|\be|<4}\frac{1}{\be!}p_{F(\be)}(\xi+\eta/2,x,\xi-\eta/2)\\
+4\sum_{|\be|=4}\frac{1}{\be!}\int_0^1 (1-\theta)^3p_{F(\be)}(\xi+\eta/2,x+\theta y,\xi-\eta/2)d\theta\Big\}dyd\eta\\
=\sum_{\be_1+\be_2=\be, |\be|<4}\frac{(-1)^{|\be_2|}}{2^{|\be|}\be_1!\be_2!}p_{F(\be)}^{(\be_1,\be_2)}(\xi,x,\xi)\\
+4(2\pi)^{-n}\sum_{\be_1+\be_2=\be,|\be|=4}\frac{(-1)^{|\be_2|}}{2^{|\be|}\be_1!\be_2!}\int \int_0^1 e^{-iy\eta}(1-\theta)^3\\
\times p_{F(\be)}^{(\be_1,\be_2)}(\xi+\eta/2,x+\theta y,\xi-\eta/2)d\theta dy d\eta\\
=\sum_{\be_1+\be_2=\be, |\be|<4}\frac{(-1)^{|\be_2|}}{2^{|\be|}\be_1!\be_2!}p_{F(\be)}^{(\be_1,\be_2)}(\xi,x,\xi)+R(x,\xi),
\end{eqnarray*}
where we have used the notation
\[
p_{F(\al)}^{(\be,\gamma)}(\eta,y,\xi)=\dif_{\eta}^{\be}\dif_{\xi}^{\gamma}D_x^{\al}p_F(\eta,y,\xi).
\]
Let us consider $P_{F(\be)}^{(\be_1,\be_2)}(\xi,x,\xi)$. Note
\begin{eqnarray*}
p_{F(\be)}^{(\be_1,\be_2)}(\xi,x,\xi)=\int F^{(\be_1)}(\xi,\zeta)p_{(\be)}(x,\zeta)F^{(\be_2)}(\xi,\zeta)d\zeta\\
=\sum \int \lr{\xi}^{-n/2}[(\zeta-\xi)\lr{\xi}^{-1/2}]^{\gamma_1}
\dif_{\sigma}^{\gamma}q((\zeta-\xi)\lr{\xi}^{-1/2})[(\zeta-\xi)\lr{\xi}^{-1/2}]^{\delta_1}\\
\times \dif_{\sigma}^{\delta}q((\zeta-\xi)\lr{\xi}^{-1/2})
\psi_{\be_1,\gamma,\gamma_1}\psi_{\be_2,\delta,\delta_1}p_{(\be)}(x,\zeta)d\zeta\\
=\sum_{|\gamma|\leq|\be_1|, |\delta|\leq |\be_2|}   \psi_{\be_1,\gamma,\gamma_1} \psi_{\be_2,\delta,\delta_1}\int \sigma^{\gamma_1}\dif_{\sigma}^{\gamma}q(\sigma) \sigma^{\delta_1}\dif_{\sigma}^{\delta}q(\sigma) p_{(\be)}(x,\xi+\sigma\lr{\xi}^{1/2})d\sigma\\
=\sum_{|\gamma|\leq|\be_1|, |\delta|\leq |\be_2|}  \psi_{\be_1,\gamma,\gamma_1}\psi_{\be_2,\delta,\delta_1}\int \sigma^{\gamma_1}\dif_{\sigma}^{\gamma}q(\sigma) \sigma^{\delta_1}\dif_{\sigma}^{\delta}q(\sigma)\Big\{
\sum_{|\al|<N}\frac{1}{\al!} p^{(\al)}_{(\be)}(x,\xi)\sigma^{\al}\lr{\xi}^{|\al|/2}\\
+N\sum_{|\al|=N}\frac{\lr{\xi}^{N/2}\sigma^{\al}}{\al!}\int_0^1 (1-\theta)^{N-1}p_{(\be)}^{(\al)}(x,\xi+\theta \sigma \lr{\xi}^{1/2})d\theta\Big\}d\sigma.
\end{eqnarray*}
Setting
\[
\begin{aligned}
\Psi_{\al,\be}=\sum_{\be_1+\be_2=\be}\sum_{ |\gamma|\leq|\be_1|, |\delta|\leq |\be_2|}\frac{(-1)^{|\be_2|}}{2^{|\be|}\be_1!\be_2!}\psi_{\be_1,\gamma,\gamma_1}\psi_{\be_2,\delta,\delta_1}\frac{\lr{\xi}^{|\al|/2}}{\al!}\\
\times\int \sigma^{\al+\gamma_1+\delta_1}\dif_{\sigma}^{\gamma}q(\sigma)\dif_{\sigma}^{\delta}q(\sigma)d\sigma,
\end{aligned}
\]
 it is clear that $\Psi_{\al,\be}\in S^{(|\al|-|\be|)/2}_{1,0}$. Then we have
\[
\begin{aligned}
\sum_{\be_1+\be_2=\be}\frac{(-1)^{|\be_2|}}{2^{|\be|}\be_1!\be_2!}p_{F(\be)}^{(\be_1,\be_2)}(\xi,x,\xi)=\sum_{|\al|<N}\Psi_{\al,\be}(\xi)p^{(\al)}_{(\be)}(x,\xi)\\
+\sum_{\be_1+\be_2=\be}\sum_{ |\gamma|\leq|\be_1|, |\delta|\leq |\be_2|}\frac{(-1)^{|\be_2|}}{2^{|\be|}\be_1!\be_2!}\psi_{\be_1,\gamma,\gamma_1}\psi_{\be_2,\delta,\delta_1}\lr{\xi}^{N/2}\\
\times \sum_{|\al|=N}\frac{1}{\al!}\int \int_0^1\sigma^{\al+\gamma_1+\delta_1}\dif_{\sigma}^{\gamma}q(\sigma)\dif_{\sigma}^{\delta}q(\sigma)\\
 \times (1-\theta)^{N-1}p_{(\be)}^{(\al)}(x,\xi+\theta \sigma \lr{\xi}^{1/2})d\theta d\sigma\\
 =\sum_{|\al|<N}\Psi_{\al,\be}(\xi)p^{(\al)}_{(\be)}(x,\xi)+\Phi_{\be,N},
\end{aligned}
\]
where $\Phi_{\be,N}\in S^{-(N+|\be|)/2}_{1/2,0}$. Hence one has 
\[
\begin{aligned}
q(x,\xi)=\sum_{|\al|<4-|\be|, |\be|<4}\Psi_{\al,\be}(\xi)p^{(\al)}_{(\be)}(x,\xi)+\sum_{|\be|<4}\Phi_{\be,4-|\be|}+R(x,\xi)\\
=\sum_{|\al+\be|<4}\Psi_{\al,\be}(\xi)p^{(\al)}_{(\be)}(x,\xi)+\sum_{|\be|<4}\Phi_{\be,4-|\be|}+R(x,\xi),
\end{aligned}
\]
where $\sum_{|\be|<4}\Phi_{\be,4-|\be|}\in S^{-2}_{1/2,0}$. Then we obtain
\[
\begin{aligned}
&\Psi_{0,0}=\psi_{0,0,0}\psi_{0,0,0}\int q(\sigma)^2d\sigma=1,\\
&\Psi_{\al,0}=\psi_{0,0,0}\psi_{0,0,0}\lr{\xi}^{1/2}\int \sigma^{\al}q(\sigma)^2d\sigma=0 \quad(|\al|=1),\\
&\Psi_{0,\be}=\sum_{\be_1+\be_2=\be}\sum_{ |\gamma|\leq|\be_1|, |\delta|\leq |\be_2|}\frac{(-1)^{|\be_2|}}{2^{|\be|}\be_1!\be_2!}\psi_{\be_1,\gamma,\gamma_1}\psi_{\be_2,\delta,\delta_1}\int \sigma^{\gamma_1+\delta_1}\dif_{\sigma}^{\gamma}q(\sigma)\dif_{\sigma}^{\delta}q(\sigma)d\sigma\\
&\qquad =\frac{1}{2}\sum_{|\gamma|\leq|\be|} \psi_{\be,\gamma,\gamma_1}\int \sigma^{\gamma_1}\dif_{\sigma}^{\gamma}q(\sigma)q(\sigma)d\sigma-\frac{1}{2}\sum_{|\delta|\leq |\be|}\psi_{\be,\delta,\delta_1}\int \sigma^{\delta_1}\dif_{\sigma}^{\delta}q(\sigma)q(\sigma)d\sigma=0\\
&\qquad =0\quad (|\be|=1).
\end{aligned}
\]
Therefore we have 
\[
q(x,\xi)= p(x,\xi)+\sum_{2\leq |\al+\be|\leq 3} \Psi_{\al,\be}p^{(\al)}_{(\be)}(x,\xi)+R(x,\xi)+S^{-2}_{1/2,0}.
\]
It remains to show that $R\in S^{-2}_{1/2,0}$. To do this we prepare the following
\begin{lem}
\label{lem:a_2}
We have
\[
|p_{F(\be)}^{(\gamma,\delta)}(\eta,x,\xi)|\leq C_{\be,\gamma,\delta}\lr{\eta}^{-|\gamma|/2}\lr{\xi}^{-|\delta|/2}.
\]
\end{lem}
\begin{proof}  By the Cauchy-Schwartz inequality we have
\begin{eqnarray*}
|p_{F(\be)}^{(\gamma,\delta)}(\eta,x,\xi)|\leq \int |F^{(\gamma)}(\eta,\zeta)p_{(\be)}(x,\zeta)F^{(\delta)}(\xi,\zeta)|d\zeta\\
\leq \Big(\int |F^{(\gamma)}(\eta,\zeta)p_{(\be)}(x,\zeta)|^2d\zeta\Big)^{1/2}\Big(\int |F^{(\delta)}(\xi,\zeta)|^2d\zeta\Big)^{1/2}.
\end{eqnarray*}
Changing the integration variables $\sigma=(\zeta-\eta)\lr{\eta}^{-1/2}$ and applying (\ref{eq:a_1}), we get
\begin{eqnarray*}
&\int |F^{(\gamma)}(\eta,\zeta)p_{(\be)}(x,\zeta)|^2d\zeta\\
&\leq \sum \int |\sigma^{\gamma_2}\dif_{\sigma}^{\gamma_1}q(\sigma)\psi_{\gamma,\gamma_1,\gamma_2}p_{(\be)}(x,\eta+\sigma\lr{\eta}^{1/2})|^2d\sigma\\
&\leq \sum C{\gamma,\gamma_i}\int|\lr{\eta}^{-(|\gamma|-|\gamma_1-\gamma_2|/2)}\sigma^{\gamma_2}\dif_{\sigma}^{\gamma_1}q(\sigma)
 p_{(\be)}(x,\eta+\sigma\lr{\eta}^{1/2})|^2d\sigma.
\end{eqnarray*}
Since  $|\gamma|\geq |\gamma_1|\geq|\gamma_1-\gamma_2|$ it is clear that
\[
\lr{\eta}^{-(|\gamma|-|\gamma_1-\gamma_2|/2)}|p_{(\be)}(x,\eta+\sigma\lr{\eta}^{1/2})|\leq C\lr{\eta}^{-|\gamma|/2}
\]
which proves
\[
\Big(\int|F^{(\gamma)}(\eta,\zeta)p_{(\be)}(x,\zeta)|^2d\zeta\Big)^{1/2}\leq C\lr{\eta}^{-|\gamma|/2}.
\]
On the other hand,
\[
\int |F^{(\delta)}(\xi,\zeta)|^2d\zeta\leq \sum \int |\sigma^{\delta_2}\dif_{\sigma}^{\delta_1}q(\sigma)\psi_{\delta,\delta_1,\delta_2}|^2d\sigma\leq C_{\delta}\lr{\xi}^{-|\delta|/2}
\]
 hence the proof is complete.  
\end{proof}

Note that
\begin{eqnarray*}
D_x^{\delta}\dif_{\xi}^{\gamma}R(x,\xi)=\sum_{|\be|=4}\sum_{\gamma'+\gamma''=\gamma}C_{\gamma'}\int (1-\theta)^3e^{-iy\eta}\lr{D_{\eta}}^{n+1}\lr{D_y}^{p}\\
\times\Big[\lr{\eta}^{-p}\lr{y}^{-n-1} p_{F(\be+\delta)}^{(\be_1+\gamma',\be_2+\gamma'')}(\xi+\frac{\eta}{2},x+\theta y,\xi-\frac{\eta}{2})\Big]dyd\eta d\theta.
\end{eqnarray*}
Consequently,  it follows from Lemma A.10 that
\begin{eqnarray*}
|D_x^{\delta}\dif_{\xi}^{\gamma}R(x,\xi)|\leq \sum \sum\sum_{|\mu|\leq p,|\nu|\leq n+1} C\int\int_0^1 |\lr{\eta}^{-p}\lr{y}^{-n-1}\\
\times p_{F(\be+\delta+\mu)}^{(\be_1+\gamma'+\nu',\be_2+\gamma''+\nu'')}(\xi+\frac{\eta}{2},x+\theta y,\xi-\frac{\eta}{2})|dyd\eta d\theta\\
\leq \sum \int \lr{\eta}^{-p}\lr{y}^{-n-1}\lr{\xi+\frac{\eta}{2}}^{-|\be_1+\gamma'+\nu'|/2}\lr{\xi-\frac{\eta}{2}}^{-|\be_2+\gamma''+\nu''|/2}dy d\eta\\
\leq \sum \int  \lr{\eta}^{-p}\lr{y}^{-n-1}\lr{\xi}^{-|\be+\gamma|/2}\lr{\eta}^{|\be+\gamma+\nu|/2}dyd\eta\leq C\lr{\xi}^{-2-|\gamma|/2}
\end{eqnarray*}
which shows $R\in S^{-2}_{1/2,0}$ and hence the assertion.  \qed

Finally, we get the following
\begin{lem}
We have
\[
{\rm Op}^w(p)=p_F-{\rm Op}^w\Big(\sum_{2\leq |\al+\be|\leq 3}\Psi_{\al,\be}(\xi)p^{(\al)}_{(\be)}(x,\xi)\Big)+{\rm Op}^w(r),
\]
where $\Psi_{\al,\be}\in S^{(|\al|-|\be|)/2}_{1,0}$ and $r\in S^{-2}_{1/2,0}$. 
\end{lem}

\end{document}